\documentclass[12pt,letter]{amsart}
\usepackage[
    top=3cm, bottom=3cm, inner=3.1cm, outer=3.1cm,
marginparwidth=2cm 
]{geometry}

\usepackage{amssymb,latexsym, amsmath, amsxtra, bbm, bm}
\usepackage[dvips]{graphics}
\usepackage{xypic}
\usepackage{verbatim}
\usepackage[abs]{overpic}
\usepackage{hyperref}
\usepackage{mathtools}

\allowdisplaybreaks[3]

\theoremstyle{plain}
        \newtheorem{theorem}{Theorem}[section]
        \newtheorem*{theorem*}{Theorem}
        \newtheorem*{conj*}{Conjecture}
        \newtheorem{lemma}[theorem]{Lemma}
        \newtheorem{prop}[theorem]{Proposition}

        \newtheorem{thmx}{Theorem}
        
\theoremstyle{definition}
        \newtheorem{definition}[theorem]{Definition}
        \newtheorem{rem}[theorem]{Remark}
         \newtheorem{rems}[theorem]{Remarks}

\theoremstyle{remark}

\numberwithin{equation}{section}
\numberwithin{theorem}{section}
\numberwithin{table}{section}
\numberwithin{figure}{section}

\providecommand{\defn}[1]{\emph{#1}}

\renewcommand{\leq}{\leqslant}

\renewcommand{\geq}{\geqslant}

\newcommand{\diam}  {\operatorname{diam}}

\newcommand{\inte}  {\operatorname{inte}}
\newcommand{\inter}  {\operatorname{int}}

\newcommand{\card} {\operatorname{card}}
\newcommand{\supp}{\operatorname{supp}}
\newcommand{\R}{\mathbb{R}}
\newcommand{\B}{\mathbb{B}}    
\newcommand{\C}{\mathbb{C}}      

\newcommand{\N}{\mathbb{N}}      
\newcommand{\Z}{\mathbb{Z}}      

\newcommand{\D}{\mathbb{D}}

\providecommand{\abs}[1]{\lvert#1\rvert}
\newcommand{\Abs}[1]{\left\lvert#1\right\rvert}
\providecommand{\Absbig}[1]{\bigl\lvert#1\bigr\rvert}
\providecommand{\Absbigg}[1]{\biggl\lvert#1\biggr\rvert}
\providecommand{\AbsBig}[1]{\Bigl\lvert#1\Bigr\rvert}

\providecommand{\norm}[1]{\|#1\|}
\providecommand{\Norm}[1]{\left\|#1\right\|}

\renewcommand{\:}{\colon}

\renewcommand{\b}{\mathfrak{b}}
\newcommand{\w}{\mathfrak{w}}
\renewcommand{\c}{\mathfrak{c}}

\newcommand{\crit}{\operatorname{crit}}
\newcommand{\post}{\operatorname{post}}

\newcommand{\CC}{\mathcal{C}}

 \newcommand{\DD}{\mathbf{D}}

\newcommand{\X} {\mathbf{X}}

\newcommand{\E} {\mathbf{E}}

\newcommand{\V} {\mathbf{V}}
 
\newcommand{\W} {\mathbf{W}}

\newcommand{\CCC}{C}

\newcommand{\Holder}[1] {\CCC^{0,#1}}

\newcommand{\Hseminorm}[2] {\Abs{#2}_{#1}}

\newcommand{\Hnorm}[3] {\Norm{#2}_{\Holder{#1}{#3}}}

\newcommand{\vertiii}[1]{{\left\vert\kern-0.25ex\left\vert\kern-0.25ex\left\vert #1 
    \right\vert\kern-0.25ex\right\vert\kern-0.25ex\right\vert}}

\newcommand{\vertiiibig}[1]{{\bigl\vert\kern-0.25ex\bigl\vert\kern-0.25ex\bigl\vert #1 
    \bigr\vert\kern-0.25ex\bigr\vert\kern-0.25ex\bigr\vert}}
    
\newcommand{\vertiiiBig}[1]{{\Bigl\vert\kern-0.25ex\Bigl\vert\kern-0.25ex\Bigl\vert #1 
    \Bigr\vert\kern-0.25ex\Bigr\vert\kern-0.25ex\Bigr\vert}}

\newcommand{\wt}[1]{\widetilde{#1}}

\newcommand{\minus}{\scalebox{0.6}[0.6]{$-\!\!\;$}}

\newcommand{\circsmall}{\scalebox{0.6}[0.6]{$\circ$}}

\newcommand{\Orb}{\mathfrak{P}}

\newcommand{\bigO}{\mathcal{O}}

\begin{document}
\title[Prime orbit theorems for expanding Thurston maps]{Prime orbit theorems for expanding Thurston maps: Genericity of strong non-integrability condition}
\author{Zhiqiang~Li \and Tianyi~Zheng}

\address{Zhiqiang~Li, School of Mathematical Sciences \& Beijing International Center for Mathematical Research, Peking University, Beijing 100871, China}
\email{zli@math.pku.edu.cn}
\address{Tianyi~Zheng, Department of Mathematics, University of California, San Diego, San Diego, CA 92093--0112}
\email{tzheng2@math.ucsd.edu}



\begin{abstract}
In the second paper \cite{LZ24b} of this series, we obtained an analog of the prime number theorem for a class of branched covering maps on the $2$-sphere $S^2$ called expanding Thurston maps, which are topological models of some non-uniformly expanding rational maps without any smoothness or holomorphicity assumption. More precisely, the number of primitive periodic orbits, ordered by a weight on each point induced by a non-constant (eventually) positive real-valued H\"{o}lder continuous function on $S^2$ satisfying the $\alpha$-strong non-integrability condition, is asymptotically the same as the well-known logarithmic integral, with an exponential error bound. In this third and last paper of the series, we show that the $\alpha$-strong non-integrability condition is generic in the class of $\alpha$-H\"{o}lder continuous functions.
\end{abstract}

\maketitle

\tableofcontents

\section{Introduction}

Complex dynamics is a vibrant field of dynamical systems, focusing on the study of iterations of polynomials and rational maps on the Riemann sphere $\widehat{\C}$. It is closely connected, via \emph{Sullivan's dictionary} \cite{Su85, Su83}, to geometric group theory, which mainly concerns the study of Kleinian groups.

In complex dynamics, the lack of uniform expansion of a rational map arises from critical points in the Julia set. Rational maps for which each critical point is preperiodic (i.e., eventually periodic) are called \emph{postcritically-finite rational maps} or \emph{rational Thurston maps}. One natural class of non-uniformly expanding rational maps are called \emph{topological Collet--Eckmann maps}, whose basic dynamical properties have been studied extensively (see for example, \cite{PRLS03, PRL07, PRL11, RLS14}). In this paper, we focus on a subclass of topological Collet--Eckmann maps for which each critical point is preperiodic and the Julia set is the whole Riemann sphere. Actually, the most general version of our results is established for topological models of these maps, called \emph{expanding Thurston maps}. Thurston maps were studied by W.~P.~Thurston in his celebrated characterization theorem of postcritically-finite rational maps among such topological models \cite{DH93}. Thurston maps and Thurson's theorem, sometimes known as a fundamental theorem of complex dynamics, are indispensable tools in the modern theory of complex dynamics. Expanding Thurston maps were studied extensively by M.~Bonk, D.~Meyer \cite{BM10, BM17} and P.~Ha\"issinsky, K.~M.~Pilgrim \cite{HP09}.

The investigations of the growth rate of the number of periodic orbits (e.g.\ closed geodesics) have been a recurring theme in dynamics and geometry. 

Inspired by the seminal works of F.~Naud \cite{Na05} and H.~Oh, D.~Winter \cite{OW17} on the growth rate of periodic orbits, known as Prime Orbit Theorems, for hyperbolic (uniformly expanding) polynomials and rational maps, we establish in \cite{LZ24b} the first Prime Orbit Theorems (to the best of our knowledge) with exponential error bounds in a non-uniformly expanding setting in complex dynamics. On the other side of Sullivan's dictionary, see related works \cite{MMO14, OW16, OP19}. For an earlier work on dynamical zeta functions for a class of sub-hyperbolic quadratic polynomials, see V.~Baladi, Y.~Jiang, and H.~H.~Rugh \cite{BJR02}. See also the related work of S.~Waddington \cite{Wad97} on strictly preperiodic points of hyperbolic rational maps and the recent work of M.~Pollicott and M.~Urba\'{n}ski \cite{PoU21} on periodic pairs and preimage points of many hyperbolic and parabolic systems.

Given a map  $f\: X \rightarrow X$ on a metric space $(X,d)$ and a function $\phi\: S^2 \rightarrow \R$,  we define the weighted length $l_{f,\phi} (\tau)$ of a primitive periodic orbit 
\begin{equation*}
\tau \coloneqq \bigl\{x, \, f(x), \,  \cdots, \,  f^{n-1}(x) \bigr\} \in \Orb(f)
\end{equation*}
as
\begin{equation}  \label{eqDefComplexLength}
l_{f,\phi} (\tau) \coloneqq \phi(x) + \phi(f(x)) + \cdots + \phi \bigl(f^{n-1}(x) \bigr).
\end{equation}
We denote by
\begin{equation}   \label{eqDefPiT}
\pi_{f,\phi}(T) \coloneqq \card \{ \tau \in \Orb(f)   :  l_{f,\phi}( \tau )  \leq T \}, \qquad   T>0,
\end{equation}
the number of primitive periodic orbits with weighted lengths up to $T$. Here $\Orb(f)$ denotes the set of all primitive periodic orbits of $f$ (see Section~\ref{sctNotation}). 

Note that the Prime Orbit Theorems in \cite{Na05, OW17} are established for the \emph{geometric potential} $\phi= \log \abs{f'}$. For hyperbolic rational maps, the Lipschitz continuity of the geometric potential plays a crucial role in \cite{Na05, OW17}. In our non-uniform expanding setting, critical points destroy the continuity of $\log\abs{f'}$. So we are left with two options to develop our theory, namely, considering 
\begin{enumerate}
\smallskip
\item[(a)] H\"older continuous $\phi$ or

\smallskip
\item[(b)] the geometric potential $\log\abs{f'}$.
\end{enumerate}
Despite the lack of H\"older continuity of $\log\abs{f'}$ in our setting, its value is closely related to the size of pull-backs of sets under backward iterations of the map $f$. This fact enables an investigation of the Prime Orbit Theorem in case (b), which will be studied in an upcoming series of separate works starting with \cite{LRL}. 

The current paper is the third and last of a series of three papers (together with \cite{LZ24a,LZ24b}) focusing on case (a), in which the incompatibility of H\"older continuity of $\phi$ and non-uniform expansion of $f$ calls for a close investigation of metric geometries associated to $f$.

We introduced the $\alpha$-strong non-integrability condition in \cite{LZ24b} for potentials for expanding Thurston maps, inspired by the work of D.~Dolgopyat \cite{Do00}. Functions satisfying the $\alpha$-strong non-integrability condition play critical roles in the following theorem established in \cite[Theorem~B]{LZ24b}.

\begin{theorem*}[Prime Orbit Theorems for rational expanding Thurston maps]  \label{thmPrimeOrbitTheorem_rational}
	Let $f\: \widehat{\C} \rightarrow \widehat{\C}$ be a postcritically-finite rational map without periodic critical points. Let $\sigma$ be the chordal metric on the Riemann sphere $\widehat{\C}$, and $\phi \: \widehat{\C} \rightarrow \R$ be an eventually positive real-valued H\"{o}lder continuous function. Then there exists a unique positive number $s_0>0$ with topological pressure $P(f,-s_0 \phi) = 0$ and there exists $N_f\in\N$ depending only on $f$ such that for each $n\in \N$ with $n\geq N_f$, the following statement holds for $F\coloneqq f^n$ and $\Phi\coloneqq \sum_{i=0}^{n-1} \phi \circ f^i$:
	
	\begin{enumerate}
		\smallskip
		\item[(i)] $\pi_{F,\Phi}(T) \sim \operatorname{Li}\bigl( e^{s_0 T} \bigr)$ as $T \to + \infty$
		if and only if $\phi$ is not cohomologous to a constant in $\CCC\bigl(\widehat{\C} \bigr)$.
		
		\smallskip
		\item[(ii)] Assume that $\phi$ satisfies the $\alpha$-strong non-integrability condition (with respect to $f$ and a visual metric) for some $\alpha \in (0,1]$. Then there exists $\delta \in (0, s_0)$ such that
		$\pi_{F,\Phi}(T) = \operatorname{Li}\bigl( e^{s_0 T} \bigr)  + \bigO \bigl( e^{(s_0 - \delta)T} \bigr) $ as $T \to + \infty$.
	\end{enumerate}
	Here $P(f,\cdot)$ denotes the topological pressure, and $\operatorname{Li} (y) \coloneqq \int_2^y\! \frac{1}{\log u} \,\mathrm{d} u$, $y>0$, is the \defn{Eulerian logarithmic integral function}.
\end{theorem*}

For the definition of eventually positive functions, see Definition~\ref{defEventuallyPositive}.

M.~Bonk, D.~Meyer \cite{BM10, BM17} and P.~Ha\"issinsky, K.~M.~Pilgrim \cite{HP09} proved that an expanding Thurston map is conjugate to a rational map if and only if the sphere $(S^2,d)$ equipped with a visual metric $d$ is quasisymmetrically equivalent to the Riemann sphere $\widehat\C$ equipped with the chordal metric. The quasisymmetry cannot be promoted to Lipschitz equivalence due to the non-uniform expansion of Thurston maps. There exist expanding Thurston maps not conjugate to rational Thurston maps (e.g.\ ones with periodic critical points). The following theorem from \cite[Theorem~C]{LZ24b} applied to all expanding Thurston maps, which form the most general setting in this series of papers.

\begin{theorem*}[Prime Orbit Theorems for expanding Thurston maps]  \label{thmPrimeOrbitTheorem}
	Let $f\: S^2 \rightarrow S^2$ be an expanding Thurston map, and $d$ be a visual metric on $S^2$ for $f$. Let $\phi \in \Holder{\alpha}(S^2,d)$ be an eventually positive real-valued H\"{o}lder continuous function with an exponent $\alpha\in (0,1]$. Denote by $s_0$ the unique positive number with topological pressure $P(f,-s_0 \phi) = 0$. Then there exists $N_f\in\N$ depending only on $f$ such that for each $n\in \N$ with $n\geq N_f$, the following statements hold for $F\coloneqq f^n$ and $\Phi\coloneqq \sum_{i=0}^{n-1} \phi \circ f^i$:
	
	\begin{enumerate}
		\smallskip
		\item[(i)] $\pi_{F,\Phi}(T) \sim \operatorname{Li}\bigl( e^{s_0 T} \bigr)$ as $T \to + \infty$ if and only if $\phi$ is not cohomologous to a constant in the space $\CCC( S^2 )$ of real-valued continuous functions on $S^2$.
		
		\smallskip
		\item[(ii)] Assume that $\phi$ satisfies the $\alpha$-strong non-integrability condition. Then there exists a constant $\delta \in (0, s_0)$ such that
		$\pi_{F,\Phi}(T) = \operatorname{Li}\bigl( e^{s_0 T} \bigr)  + \bigO \bigl( e^{(s_0 - \delta)T} \bigr)$ as $T \to + \infty$.
	\end{enumerate}
\end{theorem*}

Note that $\lim_{y\to+\infty}   \operatorname{Li}(y) / ( y / \log y )  = 1$, thus we also get $\pi_{F,\Phi}(T) \sim   e^{s_0 T} \big/ ( s_0 T)$ as $T\to + \infty$. 


The $\alpha$-strong non-integrability condition in our settings above was inspired by the work of D.~Dolgopyat \cite{Do00} on exponentially mixing flows. As mentioned in \cite[Section~1]{Do00}, the idea of D.~Dolgopyat to work with his strong non-integrability condition in \cite{Do00} goes back to the work of W.~Parry and M.~Pollicott \cite{PP97} on the stability of mixing for compact group extensions over symbolic subshifts of finite type. In \cite{Do00}, D.~Dolgopyat established that the generic suspension flows over subshifts of finite type are stably exponential mixing, but remarked (\cite[Section~1]{Do00}) that this approach does not verify the stability of exponential mixing for smooth Axiom A systems. Notably, the result of W.~Parry and M.~Pollicott also influenced the work of M.~J.~Field, I.~Melbourne, and A.~T\"orok \cite{FMT07} on the stability of mixing and rapid mixing for Axiom A systems.

In this paper, we aim to demonstrate the genericity of the $\alpha$-strong non-integrability condition in our settings, thereby establishing the genericity of the prime orbit theorems with exponential error bounds above.

The following theorem is the primary goal of the current paper. See Definition~\ref{defStrongNonIntegrability} for the notion of the $\alpha$-strong non-integrability condition.

\begin{thmx}[Genericity]     \label{thmSNIGeneric}
	Let $f\: S^2 \rightarrow S^2$ be an expanding Thurston map, and $d$ be a visual metric on $S^2$ for $f$. Fix $\alpha\in(0,1]$. The space $\Holder{\alpha}(S^2,d)$ of real-valued H\"{o}lder continuous functions with an exponent $\alpha$ is equipped with the H\"{o}lder norm $\Hnorm{\alpha}{\cdot}{(S^2,d)}$. Let $\mathcal{S}^\alpha$ be the subset of  $\Holder{\alpha}(S^2,d)$ consisting of functions satisfying the $\alpha$-strong non-integrability condition in the sense of Definition~\ref{defStrongNonIntegrability}.
	
	Then $\mathcal{S}^\alpha$ is open in $\Holder{\alpha}(S^2,d)$. Moreover, the following statements hold:
	\begin{enumerate}
		\smallskip
		\item[(i)] $\mathcal{S}^\alpha$ is an open dense subset of $\Holder{\alpha}(S^2,d)$ if $\alpha \in (0,1)$.
		
		\smallskip
		\item[(ii)] $\mathcal{S}^1$ is an open dense subset of $\Holder{1}(S^2,d)$ if the expansion factor $\Lambda$ of $d$ is not equal to the combinatorial expansion factor $\Lambda_0(f)$ of $f$.
	\end{enumerate}
\end{thmx}

The H\"{o}lder norm $\Hnorm{\alpha}{\cdot}{(S^2,d)}$ is recalled in Section~\ref{sctNotation}. The definition of the combinatorial expansion factor $\Lambda_0(f)$ of $f$ is given in Section~\ref{sctGeneric}. See \cite[Chapter~16]{BM17} for a more detailed discussion on $\Lambda_0(f)$. In particular, we always have $\Lambda \leq \Lambda_0(f)$.

We note that for each $\alpha\in(0,1]$, the subset of $\phi \in \Holder{\alpha}(S^2,d)$ that are eventually positive is open in $\Holder{\alpha}(S^2,d)$ with respect to either the uniform norm or the H\"older norm.

\smallskip

We will now give a brief description of the structure of this paper.

After fixing some notation in Section~\ref{sctNotation}, we give a review of basic definitions and results in Section~\ref{sctPreliminaries}. A constructive proof of the density of functions satisfying the $\alpha$-strong integrability condition (Theorem~\ref{thmPerturbToStrongNonIntegrable}) occupies a significant part of Section~\ref{sctDensity}. Finally, in Section~\ref{sctGeneric}, we complete the proof of Theorem~\ref{thmSNIGeneric}.

\subsection*{Acknowledgments} 
The first-named author is grateful to the Institute for Computational and Experimental Research in Mathematics (ICERM) at Brown University for the hospitality during his stay from February to May 2016, where he learned about this area of research while participating in the Semester Program ``Dimension and Dynamics'' as a research postdoctoral fellow. The authors want to express their gratitude to Mark~Pollicott for his introductory discussion on dynamical zeta functions and Prime Orbit Theorems in ICERM and many conversations since then, and to Hee~Oh for her beautiful talks on her work and helpful comments, while the first-named author also wants to thank Dennis~Sullivan for his encouragement to initiating this project, and Jianyu~Chen and  Polina~Vytnova for interesting discussions on related areas of research. Part of this work was done during the first-named author's stay at the Institute for Mathematical Sciences (IMS) at Stony Brook University as a postdoctoral fellow. He wants to thank IMS and his postdoctoral advisor Mikhail~Yu.~Lyubich for the great support and hospitality. The authors also would like to thank the anonymous referees for valuable suggestions to improve the paper. Z.~Li was partially supported by NSF DMS-1301602, DMS-1600519, NSFC Nos.~12101017, 12090010, 12090015, and BJNSF No.~1214021.

\section{Notation} \label{sctNotation}
Let $\C$ be the complex plane and $\widehat{\C}$ be the Riemann sphere. The cardinality of a set $A$ is denoted by $\card{A}$. 

Consider real-valued functions $u$, $v$, and $w$ on $(0,+\infty)$. We write $u(T) \sim v(T)$ as $T\to +\infty$ if $\lim_{T\to+\infty} \frac{u(T)}{v(T)} = 1$, and write $u(T) = v(T) + \bigO ( w(T) )$ as $T\to +\infty$ if $\limsup_{T\to+\infty} \Absbig{ \frac{u(T) - v(T) }{w(T)} } < +\infty$.

Consider a map $f\: X\rightarrow X$ on a set $X$. For each $x \in X$, we call the set $\{x, \, f(x), \, f^2(x), \dots\}$  an \emph{orbit} (starting from $x$). If an orbit has finite cardinality, then it is called a \emph{primitive periodic orbit}. The set of all primitive periodic orbits of $f$ is denoted by $\Orb(f)$.

Given a complex-valued function $\varphi\: X\rightarrow \C$, we write
\begin{equation}    \label{eqDefSnPt}
S_n \varphi (x)  = S_{n}^f \varphi (x)  \coloneqq \sum_{j=0}^{n-1} \varphi(f^j(x)) 
\end{equation}
for $x\in X$ and $n\in\N_0$. The superscript $f$ is often omitted when the map $f$ is clear from the context. Note that when $n=0$, by definition, we always have $S_0 \varphi = 0$.

Let $(X,d)$ be a metric space. For each subset $Y\subseteq X$, we denote the diameter of $Y$ by $\diam_d(Y) \coloneqq \sup\{d(x,y) : x, \, y\in Y\}$ and the interior of $Y$ by $\inter Y$. For each $r>0$ and each $x\in X$, we denote the open (resp.\ closed) ball of radius $r$ centered at $x$ by $B_d(x, r)$ (resp.\ $\overline{B_d}(x,r)$).

The space of real-valued H\"{o}lder continuous functions with an exponent $\alpha\in (0,1]$ on a compact metric space $(X,d)$ is denoted by $\Holder{\alpha}(X,d)$. For each $\psi\in\Holder{\alpha}(X,d)$, we denote
\begin{equation}   \label{eqDef|.|alpha}
\Hseminorm{\alpha,\,(X,d)}{\psi} \coloneqq \sup \{ \abs{\psi(x)- \psi(y)} / d(x,y)^\alpha  :  x, \, y\in X, \,x\neq y \},
\end{equation}
and the \emph{H\"{o}lder norm} of $\psi$ is defined as
\begin{equation}  \label{eqDefNormalizedHolderNorm}
\Hnorm{\alpha}{\psi}{(X,d)} \coloneqq  \Hseminorm{\alpha,\,(X,d)}{\psi}  + \Norm{\psi}_{\CCC^0(X)}.
\end{equation}

\section{Preliminaries}  \label{sctPreliminaries}

\subsection{Thurston maps} \label{subsctThurstonMap}
In this subsection, we go over some key concepts and results on Thurston maps, and expanding Thurston maps in particular. For a more thorough treatment of the subject, we refer to \cite{BM17}.

Let $S^2$ denote an oriented topological $2$-sphere. A continuous map $f\:S^2\rightarrow S^2$ is called a \defn{branched covering map} on $S^2$ if for each point $x\in S^2$, there exists a positive integer $d\in \N$, open neighborhoods $U$ of $x$ and $V$ of $y=f(x)$, open neighborhoods $U'$ and $V'$ of $0$ in $\widehat{\C}$, and orientation-preserving homeomorphisms $\varphi\:U\rightarrow U'$ and $\eta\:V\rightarrow V'$ such that $\varphi(x)=0$, $\eta(y)=0$, and
\begin{equation*}
(\eta\circ f\circ\varphi^{-1})(z)=z^d
\end{equation*}
for each $z\in U'$. The positive integer $d$ above is called the \defn{local degree} of $f$ at $x$ and is denoted by $\deg_f (x)$.

The \defn{degree} of $f$ is
\begin{equation}   \label{eqDeg=SumLocalDegree}
\deg f=\sum_{x\in f^{-1}(y)} \deg_f (x)
\end{equation}
for $y\in S^2$ and is independent of $y$. If $f\:S^2\rightarrow S^2$ and $g\:S^2\rightarrow S^2$ are two branched covering maps on $S^2$, then so is $f\circ g$, and
\begin{equation} \label{eqLocalDegreeProduct}
 \deg_{f\circsmall g}(x) = \deg_g(x)\deg_f(g(x)), \qquad \text{for each } x\in S^2,
\end{equation}   
and moreover, 
\begin{equation}  \label{eqDegreeProduct}
\deg(f\circ g) =  (\deg f)( \deg g).
\end{equation}

A point $x\in S^2$ is a \defn{critical point} of $f$ if $\deg_f(x) \geq 2$. The set of critical points of $f$ is denoted by $\crit f$. A point $y\in S^2$ is a \defn{postcritical point} of $f$ if $y = f^n(x)$ for some $x\in\crit f$ and $n\in\N$. The set of postcritical points of $f$ is denoted by $\post f$. Note that $\post f=\post f^n$ for all $n\in\N$.

\begin{definition} [Thurston maps] \label{defThurstonMap}
A Thurston map is a branched covering map $f\:S^2\rightarrow S^2$ on $S^2$ with $\deg f\geq 2$ and $\card(\post f)<+\infty$.
\end{definition}

We now recall the notation for cell decompositions of $S^2$ used in \cite{BM17} and \cite{Li17}. A \defn{cell of dimension $n$} in $S^2$, $n \in \{1, \, 2\}$, is a subset $c\subseteq S^2$ that is homeomorphic to the closed unit ball $\overline{\B^n}$ in $\R^n$. We define the \defn{boundary of $c$}, denoted by $\partial c$, to be the set of points corresponding to $\partial\B^n$ under such a homeomorphism between $c$ and $\overline{\B^n}$. The \defn{interior of $c$} is defined to be $\inte (c) = c \setminus \partial c$. For each point $x\in S^2$, the set $\{x\}$ is considered as a \defn{cell of dimension $0$} in $S^2$. For a cell $c$ of dimension $0$, we adopt the convention that $\partial c=\emptyset$ and $\inte (c) =c$. 

We record the following three definitions from \cite{BM17}.

\begin{definition}[Cell decompositions]\label{defcelldecomp}
Let $\DD$ be a collection of cells in $S^2$.  We say that $\DD$ is a \defn{cell decomposition of $S^2$} if the following conditions are satisfied:

\begin{itemize}

\smallskip
\item[(i)]
the union of all cells in $\DD$ is equal to $S^2$,

\smallskip
\item[(ii)] if $c\in \DD$, then $\partial c$ is a union of cells in $\DD$,

\smallskip
\item[(iii)] for $c_1, \, c_2 \in \DD$ with $c_1 \neq c_2$, we have $\inte (c_1) \cap \inte (c_2)= \emptyset$,  

\smallskip
\item[(iv)] every point in $S^2$ has a neighborhood that meets only finitely many cells in $\DD$.

\end{itemize}
\end{definition}

\begin{definition}[Refinements]\label{defrefine}
Let $\DD'$ and $\DD$ be two cell decompositions of $S^2$. We
say that $\DD'$ is a \defn{refinement} of $\DD$ if the following conditions are satisfied:
\begin{itemize}

\smallskip
\item[(i)] every cell $c\in \DD$ is the union of all cells $c'\in \DD'$ with $c'\subseteq c$,

\smallskip
\item[(ii)] for every cell $c'\in \DD'$ there exists a cell $c\in \DD$ with $c'\subseteq c$.

\end{itemize}
\end{definition}

\begin{definition}[Cellular maps and cellular Markov partitions]\label{defcellular}
Let $\DD'$ and $\DD$ be two cell decompositions of  $S^2$. We say that a continuous map $f \: S^2 \rightarrow S^2$ is \defn{cellular} for  $(\DD', \DD)$ if for every cell $c\in \DD'$, the restriction $f|_c$ of $f$ to $c$ is a homeomorphism of $c$ onto a cell in $\DD$. We say that $(\DD',\DD)$ is a \defn{cellular Markov partition} for $f$ if $f$ is cellular for $(\DD',\DD)$ and $\DD'$ is a refinement of $\DD$.
\end{definition}

Let $f\:S^2 \rightarrow S^2$ be a Thurston map, and $\CC\subseteq S^2$ be a Jordan curve containing $\post f$. Then the pair $f$ and $\CC$ induces natural cell decompositions $\DD^n(f,\CC)$ of $S^2$, for $n\in\N_0$, in the following way:

By the Jordan curve theorem, the set $S^2\setminus\CC$ has two connected components. We call the closure of one of them the \defn{white $0$-tile} for $(f,\CC)$, denoted by $X^0_\w$, and the closure of the other the \defn{black $0$-tile} for $(f,\CC)$, denoted by $X^0_\b$. The set of \defn{$0$-tiles} is $\X^0(f,\CC) \coloneqq \bigl\{ X_\b^0, \, X_\w^0 \bigr\}$. The set of \defn{$0$-vertices} is $\V^0(f,\CC) \coloneqq \post f$. We set $\overline\V^0(f,\CC) \coloneqq \{ \{x\}  :  x\in \V^0(f,\CC) \}$. The set of \defn{$0$-edges} $\E^0(f,\CC)$ is the set of the closures of the connected components of $\CC \setminus  \post f$. Then we get a cell decomposition 
\begin{equation*}
\DD^0(f,\CC) \coloneqq \X^0(f,\CC) \cup \E^0(f,\CC) \cup \overline\V^0(f,\CC)
\end{equation*}
of $S^2$ consisting of \emph{cells of level $0$}, or \defn{$0$-cells}.

We can recursively define unique cell decompositions $\DD^n(f,\CC)$, $n\in\N$, consisting of \defn{$n$-cells} such that $f$ is cellular for $(\DD^{n+1}(f,\CC),\DD^n(f,\CC))$. We refer to \cite[Lemma~5.12]{BM17} for more details. We denote by $\X^n(f,\CC)$ the set of $n$-cells of dimension 2, called \defn{$n$-tiles}; by $\E^n(f,\CC)$ the set of $n$-cells of dimension 1, called \defn{$n$-edges}; by $\overline\V^n(f,\CC)$ the set of $n$-cells of dimension 0; and by $\V^n(f,\CC)$ the set $\bigl\{x :  \{x\}\in \overline\V^n(f,\CC)\bigr\}$, called the set of \defn{$n$-vertices}. The \defn{$k$-skeleton}, for $k\in\{0, \, 1, \, 2\}$, of $\DD^n(f,\CC)$ is the union of all $n$-cells of dimension $k$ in this cell decomposition. 

We record Proposition~5.16 of \cite{BM17} here in order to summarize properties of the cell decompositions $\DD^n(f,\CC)$ defined above.

\begin{prop}[M.~Bonk \& D.~Meyer \cite{BM17}] \label{propCellDecomp}
Let $k, \, n\in \N_0$, let   $f\: S^2\rightarrow S^2$ be a Thurston map,  $\CC\subseteq S^2$ be a Jordan curve with $\post f \subseteq \CC$, and   $m=\card(\post f)$. 
 
\smallskip
\begin{itemize}

\smallskip
\item[(i)] The map  $f^k$ is cellular for $\bigl( \DD^{n+k}(f,\CC), \DD^n(f,\CC) \bigr)$. In particular, if  $c$ is any $(n+k)$-cell, then $f^k(c)$ is an $n$-cell, and $f^k|_c$ is a homeomorphism of $c$ onto $f^k(c)$.

\smallskip
\item[(ii)]  Let  $c$ be  an $n$-cell.  Then $f^{-k}(c)$ is equal to the union of all 
$(n+k)$-cells $c'$ with $f^k(c')=c$.

\smallskip
\item[(iii)] The $1$-skeleton of $\DD^n(f,\CC)$ is  equal to  $f^{-n}(\CC)$. The $0$-skeleton of $\DD^n(f,\CC)$ is the set $\V^n(f,\CC)=f^{-n}(\post f )$, and we have $\V^n(f,\CC) \subseteq \V^{n+k}(f,\CC)$. 

\smallskip
\item[(iv)] $\card(\X^n(f,\CC))=2(\deg f)^n$,  $\card(\E^n(f,\CC))=m(\deg f)^n$,  and $\card (\V^n(f,\CC)) \leq m (\deg f)^n$.

\smallskip
\item[(v)] The $n$-edges are precisely the closures of the connected components of $f^{-n}(\CC)\setminus f^{-n}(\post f )$. The $n$-tiles are precisely the closures of the connected components of $S^2\setminus f^{-n}(\CC)$.

\smallskip
\item[(vi)] Every $n$-tile  is an $m$-gon, i.e., the number of $n$-edges and the number of $n$-vertices contained in its boundary are equal to $m$.  

\smallskip
\item[(vii)] Let $F\coloneqq f^k$ be an iterate of $f$ with $k \in \N$. Then $\DD^n(F,\CC) = \DD^{nk}(f,\CC)$.
\end{itemize}
\end{prop}

We also note that for each $n$-edge $e\in\E^n(f,\CC)$, $n\in\N_0$, there exist exactly two distinct $n$-tiles $X, \, X'\in\X^n(f,\CC)$ that contain $e$.

For $n\in \N_0$, we define the \defn{set of black $n$-tiles} as
\begin{equation*}
\X_\b^n(f,\CC) \coloneqq \bigl\{X\in\X^n(f,\CC)  :  f^n(X)=X_\b^0 \bigr\},
\end{equation*}
and the \defn{set of white $n$-tiles} as
\begin{equation*}
\X_\w^n(f,\CC) \coloneqq \bigl\{X\in\X^n(f,\CC) :  f^n(X)=X_\w^0 \bigr\}.
\end{equation*}
It follows immediately from Proposition~\ref{propCellDecomp} that
\begin{equation}   \label{eqCardBlackNTiles}
\card ( \X_\b^n(f,\CC) ) = \card  (\X_\w^n(f,\CC) ) = (\deg f)^n 
\end{equation}
for each $n\in\N_0$.

From now on, if the map $f$ and the Jordan curve $\CC$ are apparent from the context, we will sometimes omit $(f,\CC)$ in the notation above.

If we fix the cell decomposition $\DD^n(f,\CC)$, $n\in\N_0$, we can define for each $v\in \V^n$ the \defn{$n$-flower of $v$} as
\begin{equation}   \label{defFlower}
W^n(v) \coloneqq \bigcup  \{\inte (c)  :  c\in \DD^n,\, v\in c \}.
\end{equation}
Note that flowers are open (in the standard topology on $S^2$). Let $\overline{W}^n(v)$ be the closure of $W^n(v)$. We define the \defn{set of all $n$-flowers} by
\begin{equation}   \label{defSetNFlower}
\W^n \coloneqq \{W^n(v)  :  v\in\V^n\}.
\end{equation}
\begin{rem}  \label{rmFlower}
For $n\in\N_0$ and $v\in\V^n$, we have 
\begin{equation*}
\overline{W}^n(v)=X_1\cup X_2\cup \cdots \cup X_m,
\end{equation*}
where $m \coloneqq 2\deg_{f^n}(v)$, and $X_1, X_2, \dots X_m$ are all the $n$-tiles that contain $v$ as a vertex (see \cite[Lemma~5.28]{BM17}). Moreover, each flower is mapped under $f$ to another flower in such a way that is similar to the map $z\mapsto z^k$ on the complex plane. More precisely, for $n\in\N_0$ and $v\in \V^{n+1}$, there exist orientation preserving homeomorphisms $\varphi\: W^{n+1}(v) \rightarrow \D$ and $\eta\: W^{n}(f(v)) \rightarrow \D$ such that $\D$ is the unit disk on $\C$, $\varphi(v)=0$, $\eta(f(v))=0$, and 
\begin{equation*}
(\eta\circ f \circ \varphi^{-1}) (z) = z^k
\end{equation*}
for all $z\in \D$, where $k \coloneqq \deg_f(v)$. Let $\overline{W}^{n+1}(v)= X_1\cup X_2\cup \cdots \cup X_m$ and $\overline{W}^n(f(v))= X'_1\cup X'_2\cup \cdots \cup X'_{m'}$, where $X_1, X_2, \dots X_m$ are all the $(n+1)$-tiles that contain $v$ as a vertex, listed counterclockwise, and $X'_1, X'_2, \dots X'_{m'}$ are all the $n$-tiles that contain $f(v)$ as a vertex, listed counterclockwise, and $f(X_1)=X'_1$. Then $m= m'k$, and $f(X_i)=X'_j$ if $i\equiv j \pmod{k}$, where $k=\deg_f(v)$. (See also Case~3 of the proof of Lemma~5.24 in \cite{BM17} for more details.) In particular, $W^n(v)$ is simply connected.
\end{rem}

We denote, for each $x\in S^2$ and $n\in\Z$,
\begin{equation}  \label{defU^n}
U^n(x) \coloneqq \bigcup \{Y^n\in \X^n  :     \text{there exists } X^n\in\X^n  
                                        \text{ with } x\in X^n, \, X^n\cap Y^n \neq \emptyset  \}  
\end{equation}
if $n\geq 0$, and set $U^n(x) \coloneqq S^2$ otherwise. 

We can now give a definition of expanding Thurston maps.

\begin{definition} [Expansion] \label{defExpanding}
A Thurston map $f\:S^2\rightarrow S^2$ is called \defn{expanding} if there exists a metric $d$ on $S^2$ that induces the standard topology on $S^2$ and a Jordan curve $\CC\subseteq S^2$ containing $\post f$ such that 
\begin{equation*}
\lim_{n\to+\infty}\max \{\diam_d(X)  :  X\in \X^n(f,\CC)\}=0.
\end{equation*}
\end{definition}

\begin{rems}  \label{rmExpanding}
It is clear from Proposition~\ref{propCellDecomp}~(vii) and Definition~\ref{defExpanding} that if $f$ is an expanding Thurston map, so is $f^n$ for each $n\in\N$. We observe that being expanding is a topological property of a Thurston map and independent of the choice of the metric $d$ that generates the standard topology on $S^2$. By Lemma~6.2 in \cite{BM17}, it is also independent of the choice of the Jordan curve $\CC$ containing $\post f$. More precisely, if $f$ is an expanding Thurston map, then
\begin{equation*}
\lim_{n\to+\infty}\max \!\bigl\{ \! \diam_{\wt{d}}(X)  :  X\in \X^n\bigl(f,\wt\CC \hspace{0.5mm}\bigr)\hspace{-0.3mm} \bigr\}\hspace{-0.3mm}=0,
\end{equation*}
for each metric $\wt{d}$ that generates the standard topology on $S^2$ and each Jordan curve $\wt\CC\subseteq S^2$ that contains $\post f$.
\end{rems}

P.~Ha\"{\i}ssinsky and K.~M.~Pilgrim developed a notion of expansion in a more general context for finite branched coverings between topological spaces (see \cite[Sections~2.1 and~2.2]{HP09}). This applies to Thurston maps and their notion of expansion is equivalent to our notion defined above in the context of Thurston maps (see \cite[Proposition~6.4]{BM17}). Such concepts of expansion are natural analogs, in the contexts of finite branched coverings and Thurston maps, to some of the more classical versions, such as expansive homeomorphisms and forward-expansive continuous maps between compact metric spaces (see for example, \cite[Definition~3.2.11]{KH95}), and distance-expanding maps between compact metric spaces (see for example, \cite[Chapter~4]{PrU10}). Our notion of expansion is not equivalent to any such classical notion in the context of Thurston maps. One topological obstruction comes from the presence of critical points for (non-homeomorphic) branched covering maps on $S^2$. In fact, as mentioned in the introduction, there are subtle connections between our notion of expansion and some classical notions of weak expansion. More precisely, one can prove that an expanding Thurston map is asymptotically $h$-expansive if and only if it has no periodic points. Moreover, such a map is never $h$-expansive. Asymptotic $h$-expansiveness and $h$-expansiveness are two notions of weak expansion introduced by M.~Misiurewicz \cite{Mi73} and R.~Bowen \cite{Bow72}, respectively. See \cite{Li15} for details.

For an expanding Thurston map $f$, we can fix a particular metric $d$ on $S^2$ called a \emph{visual metric for $f$}. For the existence and properties of such metrics, see \cite[Chapter~8]{BM17}. For a visual metric $d$ for $f$, there exists a unique constant $\Lambda > 1$ called the \emph{expansion factor} of $d$ (see \cite[Chapter~8]{BM17} for more details). One major advantage of a visual metric $d$ is that in $(S^2,d)$, we have good quantitative control over the sizes of the cells in the cell decompositions discussed above. We summarize several results of this type (\cite[Proposition~8.4, Lemmas~8.10,~8.11]{BM17}) in the lemma below.

\begin{lemma}[M.~Bonk \& D.~Meyer \cite{BM17}]   \label{lmCellBoundsBM}
Let $f\:S^2 \rightarrow S^2$ be an expanding Thurston map, and $\CC \subseteq S^2$ be a Jordan curve containing $\post f$. Let $d$ be a visual metric on $S^2$ for $f$ with expansion factor $\Lambda>1$. Then there exist constants $C\geq 1$, $K\geq 1$, and $n_0\in\N_0$ with the following properties:
\begin{enumerate}
\smallskip
\item[(i)] $d(\sigma,\tau) \geq C^{-1} \Lambda^{-n}$ whenever $\sigma$ and $\tau$ are disjoint $n$-cells for $n\in \N_0$.

\smallskip
\item[(ii)] $C^{-1} \Lambda^{-n} \leq \diam_d(\tau) \leq C\Lambda^{-n}$ for all $n$-edges and all $n$-tiles $\tau$ for $n\in\N_0$.

\smallskip
\item[(iii)] $B_d(x,K^{-1} \Lambda^{-n} ) \subseteq U^n(x) \subseteq B_d(x, K\Lambda^{-n})$ for $x\in S^2$ and $n\in\N_0$.

\smallskip
\item[(iv)] $U^{n+n_0} (x)\subseteq B_d(x,r) \subseteq U^{n-n_0}(x)$ where $n= \lceil -\log r / \log \Lambda \rceil$ for $r>0$ and $x\in S^2$.

\smallskip
\item[(v)] For every $n$-tile $X^n\in\X^n(f,\CC)$, $n\in\N_0$, there exists a point $p\in X^n$ such that $B_d(p,C^{-1}\Lambda^{-n}) \subseteq X^n \subseteq B_d(p,C\Lambda^{-n})$.
\end{enumerate}

Conversely, if $\wt{d}$ is a metric on $S^2$ satisfying conditions \textnormal{(i)} and \textnormal{(ii)} for some constant $C\geq 1$, then $\wt{d}$ is a visual metric with expansion factor $\Lambda>1$.
\end{lemma}

Recall that $U^n(x)$ is defined in (\ref{defU^n}).

Note that a visual metric $d$ induces the standard topology on $S^2$ (\cite[Proposition~8.3]{BM17}).

In fact, visual metrics play a crucial role in connecting the dynamical arguments with geometric properties for rational expanding Thurston maps, especially Latt\`{e}s maps.

A Jordan curve $\CC\subseteq S^2$ is \defn{$f$-invariant} if $f(\CC)\subseteq \CC$. We are interested in $f$-invariant Jordan curves that contain $\post f$, since for such a Jordan curve $\CC$, we get a cellular Markov partition $(\DD^1(f,\CC),\DD^0(f,\CC))$ for $f$. According to Example~15.11 in \cite{BM17}, such $f$-invariant Jordan curves containing $\post{f}$ need not exist. However, M.~Bonk and D.~Meyer \cite[Theorem~15.1]{BM17} proved that there exists an $f^n$-invariant Jordan curve $\CC$ containing $\post{f}$ for each sufficiently large $n$ depending on $f$. A slightly stronger version of this result was proved in \cite[Lemma~3.11]{Li16}, and we record it below.

\begin{lemma}[M.~Bonk \& D.~Meyer \cite{BM17}, Z.~Li \cite{Li16}]  \label{lmCexistsL}
Let $f\:S^2\rightarrow S^2$ be an expanding Thurston map, and $\wt{\CC}\subseteq S^2$ be a Jordan curve with $\post f\subseteq \wt{\CC}$. Then there exists an integer $N(f,\wt{\CC}) \in \N$ such that for each $n\geq N(f,\wt{\CC})$ there exists an $f^n$-invariant Jordan curve $\CC$ isotopic to $\wt{\CC}$ rel.\ $\post f$ such that no $n$-tile in $\X^n(f,\CC)$ joins opposite sides of $\CC$.
\end{lemma}

The phrase ``joining opposite sides'' has a specific meaning in our context. 

\begin{definition}[Joining opposite sides]  \label{defJoinOppositeSides} 
Fix a Thurston map $f$ with $\card(\post f) \geq 3$ and an $f$-invariant Jordan curve $\CC$ containing $\post f$.  A set $K\subseteq S^2$ \defn{joins opposite sides} of $\CC$ if $K$ meets two disjoint $0$-edges when $\card( \post f)\geq 4$, or $K$ meets  all  three $0$-edges when $\card(\post f)=3$. 
 \end{definition}
 
Note that $\card (\post f) \geq 3$ for each expanding Thurston map $f$ \cite[Lemma~6.1]{BM17}.

The following lemma proved in \cite[Lemma~3.13]{Li18} generalizes \cite[Lemma~15.25]{BM17}.

\begin{lemma}[M.~Bonk \& D.~Meyer \cite{BM17}, Z.~Li \cite{Li18}]   \label{lmMetricDistortion}
Let $f\:S^2 \rightarrow S^2$ be an expanding Thurston map, and $\CC \subseteq S^2$ be a Jordan curve that satisfies $\post f \subseteq \CC$ and $f^{n_\CC}(\CC)\subseteq\CC$ for some $n_\CC\in\N$. Let $d$ be a visual metric on $S^2$ for $f$ with expansion factor $\Lambda>1$. Then there exists a constant $C_0 > 1$, depending only on $f$, $d$, $\CC$, and $n_\CC$, with the following property:

If $k, \, n\in\N_0$, $X^{n+k}\in\X^{n+k}(f,\CC)$, and $x, \, y\in X^{n+k}$, then 
\begin{equation}   \label{eqMetricDistortion}
C_0^{-1} d(x,y) \leq \Lambda^{-n}  d(f^n(x),f^n(y)) \leq C_0 d(x,y).
\end{equation}
\end{lemma}

The following distortion lemma serves as a cornerstone in the development of thermodynamic formalism for expanding Thurston maps in \cite{Li18} (see \cite[Lemma~5.1]{Li18}).

\begin{lemma}[Z.~Li \cite{Li18}]    \label{lmSnPhiBound}
Let $f\:S^2 \rightarrow S^2$ be an expanding Thurston map and $\CC \subseteq S^2$ be a Jordan curve containing $\post f$ with the property that $f^{n_\CC}(\CC)\subseteq \CC$ for some $n_\CC\in\N$. Let $d$ be a visual metric on $S^2$ for $f$ with expansion factor $\Lambda>1$. Let $\phi\in \Holder{\alpha}(S^2,d)$ be a real-valued H\"{o}lder continuous function with an exponent $\alpha\in(0,1]$. Then there exists a constant $C_1=C_1(f,\CC,d,\phi,\alpha)$ depending only on $f$, $\CC$, $d$, $\phi$, and $\alpha$ such that
\begin{equation}  \label{eqSnPhiBound}
\Abs{S_n\phi(x)-S_n\phi(y)}  \leq C_1 d(f^n(x),f^n(y))^\alpha,
\end{equation}
for $n, \, m\in\N_0$ with $n\leq m $, $X^m\in\X^m(f,\CC)$, and $x, \, y\in X^m$. Quantitatively, we can choose 
\begin{equation}   \label{eqC1Expression}
C_1 \coloneqq  \Hseminorm{\alpha,\, (S^2,d)}{\phi} C_0  ( 1-\Lambda^{-\alpha} )^{-1},
\end{equation}
where $C_0 > 1$ is the constant depending only on $f$, $\CC$, and $d$ from Lemma~\ref{lmMetricDistortion}.
\end{lemma}

\begin{definition}[Eventually positive functions]  \label{defEventuallyPositive}
Let $g\: X\rightarrow X$ be a map on a set $X$, and $\varphi\:X\rightarrow\C$ be a complex-valued function on $X$. Then $\varphi$ is \defn{eventually positive} if there exists $N\in\N$ such that $S_n\varphi(x)>0$ for each $x\in X$ and each $n\in\N$ with $n\geq N$.
\end{definition}

\subsection{Combinatorial expansion factor}
We first recall some concepts related to the expansion of expanding Thurston maps from a combinatorial point of view. Suppose that $f\: S^2 \rightarrow S^2$ is a Thurston map and $\CC\subseteq S^2$ is a Jordan curve with $\post f \subseteq \CC$. For each $n\in\N_0$, we denote by $D_n(f,\CC)$ the minimal number of $n$-tiles required to form a connected set joining opposite sides of $\CC$; more precisely,
\begin{align}   \label{eqDefDn}
D_n(f,\CC) \coloneqq \min \biggl\{ N\in\N  &:   \text{there exist } X_1,  \, X_2,  \, \dots,  \, X_N \in \X^n(f,\CC)  \text{ such that } \\
                                                                                            &\;\;\; \bigcup_{j=1}^{N} X_j \text{ is connected and joins opposite sides of } \CC \biggr\}.  \notag 
\end{align}
See \cite[Section~5.7]{BM17} for more properties of $D_n(f,\CC)$. M.~Bonk and D.~Meyer showed in \cite[Proposition~16.1]{BM17} that the limit
\begin{equation}  \label{eqDefCombExpansionFactor}
\Lambda_0(f) \coloneqq \lim_{n\to+\infty} D_n(f,\CC)^{1/n}
\end{equation}
exists and is independent of $\CC$. We have $\Lambda_0(f) \in (1,+\infty)$. The constant $\Lambda_0(f)$ is called the \defn{combinatorial expansion factor} of $f$.

The combinatorial expansion factor $\Lambda_0(f)$ serves as a sharp upper bound for the expansion factors of visual metrics of $f$; more precisely, for an expanding Thurston map $f$, the following statements hold (\cite[Theorem~16.3]{BM17}):
\begin{enumerate}
\smallskip
\item[(i)] If $\Lambda$ is the expansion factor of a visual metric for $f$, then $\Lambda\in (1, \Lambda_0(f) ]$.

\smallskip
\item[(ii)] Conversely, if $\Lambda \in (1, \Lambda_0(f) )$, then there exists a visual metric for $f$ with expansion factor $\Lambda$.
\end{enumerate}

\subsection{Strong non-integrability condition}

We recall the strong non-integrability condition from \cite[Subsection~7.1]{LZ24b}.

\begin{definition}[Strong non-integrability condition]   \label{defStrongNonIntegrability}
Let $f\: S^2 \rightarrow S^2$ be an expanding Thurston map and $d$ be a visual metric on $S^2$ for $f$. Fix $\alpha \in (0,1]$. Let $\phi\in \Holder{\alpha}(S^2,d)$ be a real-valued H\"{o}lder continuous function with an exponent $\alpha$. 

\begin{enumerate}

\smallskip
\item[(1)] We say that $\phi$ satisfies the \defn{$(\CC, \alpha)$-strong non-integrability condition} (with respect to $f$ and $d$), for a Jordan curve $\CC\subseteq S^2$ with $\post f \subseteq \CC$, if there exist 
\begin{enumerate}
	\smallskip
	\item[(a)] numbers $N_0,\, M_0\in\N$, $\varepsilon \in (0,1)$, and
	
	\smallskip
	\item[(b)] $M_0$-tiles $Y^{M_0}_\b\in \X^{M_0}_\b(f,\CC)$,  $Y^{M_0}_\w\in \X^{M_0}_\w(f,\CC)$
\end{enumerate}
such that for each $\c\in\{\b, \, \w\}$, each integer $M \geq M_0$, and each $M$-tile $X\in \X^M(f,\CC)$ with $X \subseteq Y^{M_0}_\c$, there exist two points $x_1(X),\, x_2(X) \in X$ with the following properties:
\begin{enumerate}
\smallskip
\item[(i)] $\min \{ d(x_1(X), S^2 \setminus X), \, d(x_2(X), S^2 \setminus X), d(x_1(X), x_2(X)) \} \geq \varepsilon \diam_d(X)$, and

\smallskip
\item[(ii)] for each integer $N \geq N_0$, there exist two  $(N+M_0)$-tiles $X^{N+M_0}_{\c,1}, \, X^{N+M_0}_{\c,2} \in \X^{N+M_0}(f,\CC)$ such that $Y^{M_0}_\c = f^N\bigl(  X^{N+M_0}_{\c,1}  \bigr) =   f^N\bigl(  X^{N+M_0}_{\c,2}  \bigr)$, and that
\begin{equation}   \label{eqSNIBoundsDefn}
\frac{ \Abs{  S_{N }\phi ( \varsigma_1 (x_1(X)) ) -  S_{N }\phi ( \varsigma_2 (x_1(X)) )  -S_{N }\phi ( \varsigma_1 (x_2(X)) ) +  S_{N }\phi ( \varsigma_2 (x_2(X)) )   }  } {d(x_1(X),x_2(X))^\alpha}
\geq \varepsilon,
\end{equation}
where we write $\varsigma_1 \coloneqq \bigl(f^{N}\big|_{X^{N+M_0}_{\c,1}} \bigr)^{-1}$ and $\varsigma_2 \coloneqq \bigl(f^{N}\big|_{X^{N+M_0}_{\c,2}} \bigr)^{-1}$. 
\end{enumerate}

\smallskip
\item[(2)] We say that $\phi$ satisfies the \defn{$\alpha$-strong non-integrability condition} (with respect to $f$ and $d$) if $\phi$ satisfies the $(\CC, \alpha)$-strong non-integrability condition with respect to $f$ and $d$ for some  Jordan curve $\CC\subseteq S^2$ with $\post f \subseteq \CC$.

\smallskip
\item[(3)] We say that $\phi$ satisfies the \defn{strong non-integrability condition} (with respect to $f$ and $d$) if $\phi$ satisfies the $\alpha'$-strong non-integrability condition with respect to $f$ and $d$ for some  $\alpha' \in (0, \alpha]$.
\end{enumerate}
\end{definition}

We have shown in \cite[Lemma~7.2]{LZ24b} that the strong non-integrability condition is independent of the Jordan curve $\CC$. We record it here for the convenience of the reader.

\begin{lemma}   \label{lmSNIwoC}
Let $f\: S^2 \rightarrow S^2$ be an expanding Thurston map and $d$ be a visual metric on $S^2$ for $f$. Let $\CC$ and $\widehat\CC$ be Jordan curves on $S^2$ with $\post f \subseteq \CC \cap \widehat\CC$. Let $\phi \in \Holder{\alpha}(S^2,d)$ be a real-valued H\"{o}lder continuous function with an exponent $\alpha \in (0,1]$. Fix arbitrary integers $n,\, \widehat{n} \in \N$. Let $F\coloneqq f^n$ and $\widehat{F} \coloneqq f^{\widehat{n}}$ be iterates of $f$. Then $\Phi \coloneqq S_n^f \phi$ satisfies the $(\CC, \alpha)$-strong non-integrability condition with respect to $F$ and $d$ if and only if $\widehat\Phi \coloneqq S_{\widehat{n}}^f \phi$ satisfies the $( \widehat\CC, \alpha )$-strong non-integrability condition with respect to $\widehat{F}$ and $d$.

In particular, if $\phi$ satisfies the $\alpha$-strong non-integrability condition with respect to $f$ and $d$, then it satisfies the $(\CC, \alpha)$-strong non-integrability condition with respect to $f$ and $d$.
\end{lemma}

Denote
\begin{equation}   \label{eqDefSigma-}
	\Sigma_{f,\,\CC}^- \coloneqq \bigl\{\{X_{\minus i}\}_{i\in\N_0}  :  X_{\minus i}\in \X^1(f,\CC) \text{ and } f \bigl(X_{\minus (i+1)}\bigr) \supseteq X_{\minus i},\text{ for } i\in\N_0 \bigr\}.
\end{equation}

The next notion is crucial in Section~\ref{sctDensity}.

\begin{definition}[Temporal distance]  \label{defTemporalDist}
	Let $f\: S^2 \rightarrow S^2$ be an expanding Thurston map and $d$ be a visual metric on $S^2$ for $f$. Let $\CC$ be a Jordan curve on $S^2$ with $\post f \subseteq \CC$ and $f(\CC)\subseteq \CC$. Let $\phi \in \Holder{\alpha}(S^2,d)$ be a real-valued H\"{o}lder continuous function with an exponent $\alpha \in (0,1]$. 
	
	For $\xi=\{ \xi_{\minus i} \}_{i\in\N_0} \in  \Sigma_{f,\,\CC}^-$ and $\eta=\{ \eta_{\minus i} \}_{i\in\N_0} \in  \Sigma_{f,\,\CC}^-$ with $f(\xi_0) = f(\eta_0)$, we define the \defn{temporal distance} $\psi^{f,\,\CC}_{\xi,\,\eta}$ as 
	\begin{equation*}
	\psi^{f,\,\CC}_{\xi,\,\eta}(x,y) \coloneqq \Delta^{f,\,\CC}_{\psi,\,\xi} (x,y) - \Delta^{f,\,\CC}_{\psi,\,\eta} (x,y)
	\end{equation*}
	for each
	$
	(x,y)\in \bigcup\limits_{\substack{X\in\X^1(f,\CC) \\ X\subseteq f(\xi_0)}}X \times X.
	$
\end{definition}

For the definition and characterizations of a qualitative version of the strong non-integrability condition, formulated in terms of the temporal distance, see \cite[Definition~7.3 and Theorem~F]{LZ24a}.

\section{A constructive proof of density}   \label{sctDensity}

The main result of this section is Theorem~\ref{thmPerturbToStrongNonIntegrable}. I first need to establish the following lemma.

\begin{lemma}   \label{lmDisjointBackwardOrbits}
Let $f\: S^2 \rightarrow S^2$ be an expanding Thurston map with a Jordan curve $\CC\subseteq S^2$ satisfying $\post f\subseteq \CC$ and $f(\CC)\subseteq \CC$. Then there exist two sequences of $1$-tiles $\{ \xi_{\minus i} \}_{i\in\N_0}, \, \{ \xi'_{\minus i'} \}_{i'\in\N_0}  \in \Sigma_{f,\,\CC}^-$ such that $f (\xi_0 ) = f (\xi'_0)$ and $\xi_{\minus i} = \xi_0 \neq \xi'_{\minus i'}$ for all $i, \, i'\in\N_0$.
\end{lemma}

\begin{proof}
We first claim that if the white $0$-tile $X^0_\w \in\X^0$ does not contain a white $1$-tile, then there exists a black $1$-tile $X^1_\b \in \X^1_\b$ such that $X^1_\b = X^0_\w$.

Indeed, note that for each $1$-edge $e^1\in\E^1$, there exists a unique black $1$-tile $X_\b\in \X^1_\b$ and a unique white $1$-tile $X_\w\in \X^1_\w$ such that $X_\b \cap X_\w = e^1$. Suppose that $X^0_\w$ is a union  $X^0_\w = \bigcup_{i=1}^k X_i$ of $k$ distinct black $1$-tiles $X_i \in \X^1_\b$, $i\in\{1, \, 2, \, \dots, \, k\}$, then $\bigcup_{i=1}^k \partial X_i \subseteq \partial X^0_\w = \CC$. Since each of $\CC$ and $\partial X_i$, $i\in\{1, \, 2, \, \dots, \, k\}$, is a Jordan curve and $\partial X_j \neq \partial X_{j'}$ for $1 \leq j < j' \leq k$, we conclude that $k=1$, establishing the claim.

\smallskip

A similar statement holds if we exchange black and white.

Next, we observe that if the white $0$-tile $X^0_\w$ is also a white $1$-tile or the black $0$-tile $X^0_\b$ is also a black $1$-tile, then $f$ cannot be expanding.

Hence, it suffices to construct the sequences $\{ \xi_{\minus i} \}_{i\in\N_0}$ and $\{ \xi'_{\minus i'} \}_{i'\in\N_0}$ in the following two cases:

\smallskip

\emph{Case 1.} Either $X^0_\w = X^1_\b$ for some black $1$-tile $X^1_\b \in \X^1_b$ or $X^0_\b = X^1_\w$ for some white $1$-tile $X^1_\w \in \X^1_w$. Without loss of generality, we assume the former holds. Since $\deg f \geq 2$, we can choose a black $1$-tile $Y^1_\b \in \X^1_\b$ and a white $1$-tile $Y^1_\w \in \X^1_\w$ such that $Y^1_\b \cup Y^1_\w \subseteq X^0_\b$. Then we define $\xi_{\minus i} \coloneqq Y^1_\b$ for all $i\in\N_0$, $\xi'_{\minus i'} \coloneqq X^1_\b$ if $i'\in\N_0$ is even, and $\xi'_{\minus i'} \coloneqq Y^1_\w$ if $i'\in\N_0$ is odd.

\smallskip

\emph{Case 2.} There exist black $1$-tiles $X^1_\b, \, Y^1_\b \in \X^1_\b$ and white $1$-tiles $X^1_\w,  \, Y^1_\w \in \X^1_\w$ such that $X^1_\b \cup X^1_\w \subseteq X^0_\w$ and $Y^1_\b \cup Y^1_\w \subseteq X^0_\b$. Then we define $\xi_{\minus i} \coloneqq Y^1_\b$ for all $i\in\N_0$, $\xi'_0 \coloneqq X^1_\b$, and $\xi'_{\minus i'} \coloneqq X^1_\w$ for all $i'\in\N$.

\smallskip

It is trivial to check that in both cases, $\{ \xi_{\minus i} \}_{i\in\N_0},  \,  \{ \xi'_{\minus i'} \}_{i'\in\N_0}  \in \Sigma_{f,\,\CC}^-$,  $f  (\xi_0  ) = f  (\xi'_0 )$, and $\xi_{\minus i} = \xi_0 \neq \xi'_{\minus i'}$ for all $i, \, i'\in\N_0$.
\end{proof}

\begin{theorem}     \label{thmPerturbToStrongNonIntegrable}
Let $f\: S^2 \rightarrow S^2$ be an expanding Thurston map with a Jordan curve $\CC\subseteq S^2$ satisfying $\post f\subseteq \CC$ and $f(\CC)\subseteq \CC$. Let $d$ be a visual metric on $S^2$ for $f$ with expansion factor $\Lambda>1$. Fix $\alpha\in(0,1]$. Assume that $\Lambda^\alpha < \Lambda_0(f)$. Then there exists a constant $C_{\sharp} > 0$ such that for each $\varepsilon>0$ and each  real-valued H\"{o}lder continuous function $\varphi \in \Holder{\alpha}(S^2,d)$ with an exponent $\alpha$, there exist integers $N_0, \, M_0\in \N$, $M_0$-tiles $Y^{M_0}_\b \in \X^{M_0}_\b(f,\CC)$, $Y^{M_0}_\w \in \X^{M_0}_\w(f,\CC)$, and a real-valued H\"{o}lder continuous function $\phi \in \Holder{\alpha}(S^2,d)$ such that for each $\c\in\{\b, \, \w\}$, each integer $M\geq M_0$, and each $M$-tile $X\in\X^M(f,\CC)$ with $X\subseteq Y^{M_0}_\c$, there exist two points $x_1(X),\, x_2(X) \in X$ with the following properties:
\begin{enumerate}
\smallskip
\item[(i)] $\min \{ d(x_1(X), S^2 \setminus X), \, d(x_2(X), S^2 \setminus X), \, d(x_1(X), x_2(X)) \} \geq \varepsilon \diam_d(X)$.

\smallskip
\item[(ii)] for each integer $N' \geq N_0$, there exist two  $(N'+M_0)$-tiles $X^{N'+M_0}_{\c,1},\,  X^{N'+M_0}_{\c,2} \in \X^{N'+M_0}(f,\CC)$ such that $Y^{M_0}_\c = f^{N'}\bigl(  X^{N'+M_0}_{\c,1}  \bigr) =   f^{N'}\bigl(  X^{N'+M_0}_{\c,2}  \bigr)$, and that
\begin{equation}   \label{eqSNIBoundsPerturb}
\frac{ \Abs{  S_{N' }\phi ( \varsigma_1 (x_1(X)) ) -  S_{N' }\phi ( \varsigma_2 (x_1(X)) )  -S_{N' }\phi ( \varsigma_1 (x_2(X)) ) +  S_{N' }\phi ( \varsigma_2 (x_2(X)) )   }  } {d(x_1(X),x_2(X))^\alpha}
\geq \varepsilon,
\end{equation}
where we write $\varsigma_1 \coloneqq \bigl(f^{N'}\big|_{X^{N'+M_0}_{\c,1}} \bigr)^{-1}$ and $\varsigma_2 \coloneqq \bigl(f^{N'}\big|_{X^{N'+M_0}_{\c,2}} \bigr)^{-1}$. 

\smallskip
\item[(iii)] $\Hnorm{\alpha}{\phi - \varphi}{(S^2,d)} \leq C_{\sharp} \varepsilon$.
\end{enumerate}
\end{theorem}

\begin{proof}

Denote
\begin{equation}   \label{eqDefC26}
C_{\dagger} \coloneqq 4 C^\alpha \Lambda^\alpha > 1.
\end{equation}
Here $C\geq 1$ is the constant from Lemma~\ref{lmCellBoundsBM} depending only on $f$, $\CC$, and $d$.

Since $\Lambda^\alpha < \Lambda_0(f) = \lim_{n\to+\infty} D_n(f,\CC)^{1/n}$ (see (\ref{eqDefCombExpansionFactor})), we can fix $N\in\N$ large enough such that the following statements hold:
\begin{itemize}
\smallskip
\item $3< 3  C_{\dagger} C<\Lambda^{\alpha N} < D_N(f,\CC) - 1$.

\smallskip
\item There exist $u^1_\b,\, u^2_\b,\, u^1_\w,\, u^2_\w \in \V^N$ such that for all $\c \in \{ \b,  \, \w \}$,
\begin{equation}   \label{eqPfthmPerturbToStrongNonIntegrable_Flower}
\overline{W}^N \bigl( u^1_\c \bigr) \cup  \overline{W}^N \bigl( u^2_\c \bigr)   \subseteq \inte \bigl( X^0_\c \bigr)  \quad \mbox{ and }\quad    
\overline{W}^N \bigl( u^1_\c \bigr) \cap  \overline{W}^N \bigl( u^2_\c \bigr)    = \emptyset.   
\end{equation}
\end{itemize}

We denote $D_N \coloneqq D_N(f,\CC)$ in the remaining part of this proof.

It suffices to establish the theorem for $\varepsilon >0$ sufficiently small. Fix arbitrary
\begin{equation}   \label{eqPfthmPerturbToStrongNonIntegrable_varepsilon}
\varepsilon \in \bigl(0, C^{-2} \Lambda^{ - 2N } \bigr)  \subseteq (0,1).
\end{equation} 

We define the following constants
\begin{align}
\rho    & \coloneqq  \Lambda^{\alpha N} ( D_N - 1 )^{-1} \in (0,1), \label{eqPfthmPerturbToStrongNonIntegrable_rho}  \\
C_{\sharp} & \coloneqq 1 + C_{\dagger}   C \bigl( 4   (1-\rho)^{-1} + \Lambda^{ \alpha N} \bigl( 1 - \Lambda^{ - \alpha N} \bigr)^{-1} \bigr),  \label{eqDefC27}  \\
N_0    & \coloneqq  \bigl\lceil  \alpha^{ - 1 } \log_\Lambda \bigl( 2 C^2 \varepsilon^{-1 - \alpha} (\Hnorm{\alpha}{\varphi}{(S^2,d)} + \varepsilon C_{\sharp} ) C_0 / ( 1-\Lambda^{-\alpha} ) \bigr)  \bigr\rceil.    \label{eqPfthmPerturbToStrongNonIntegrable_N0}  
\end{align}
Here $C_0 > 1$ is the constant depending only on $f$, $\CC$, and $d$ from Lemma~\ref{lmMetricDistortion}.

Choose two sequences of $1$-tiles $\xi \coloneqq \{ \xi_{\minus i} \}_{i\in\N_0}  \in \Sigma_{f,\,\CC}^-$ and $\xi' \coloneqq \{ \xi'_{ \minus i'} \}_{i'\in\N_0}  \in \Sigma_{f,\,\CC}^-$ as in Lemma~\ref{lmDisjointBackwardOrbits} such that $f  ( \xi_0  ) = f  ( \xi'_0 )$ and $\xi_{\minus i} = \xi_0 \neq \xi'_{\minus i'}$ for all $i, \, i'\in\N_0$. We denote, for each $j\in\N$,
\begin{equation}  \label{eqPfthmPerturbToStrongNonIntegrable_tau}
\tau_j \coloneqq \bigl( f|_{\xi_{1-j}} \bigr)^{-1} \circ \cdots \circ ( f|_{\xi_{ \minus 1}} )^{-1} \circ ( f|_{\xi_{0}} )^{-1} \text{ and }
\tau'_j \coloneqq \bigl( f|_{\xi'_{1-j}} \bigr)^{-1} \circ \cdots \circ \bigl( f|_{\xi'_{ \minus 1}} \bigr)^{-1} \circ \bigl( f|_{\xi'_{0}} \bigr)^{-1}.
\end{equation}

Since $f$ is an expanding Thurston map and $\xi_0$ is a $1$-tile, we have $f  ( \xi_0  ) \supsetneq \xi_0$, for otherwise $\xi_0$ would have been an $n$-tile for each $n\in\N_0$. Thus, we can fix a constant
\begin{equation}  \label{eqPfthmPerturbToStrongNonIntegrable_M0} 
M_0  \geq    \alpha^{-1} \log_\Lambda \bigl( 2  C_{\dagger} \big/ \bigl( 1 - \Lambda^{-\alpha N} \bigr) \bigr)  
\end{equation}
large enough such that we can choose $Y^{M_0}_\b \in \X^{M_0}_\b$ and $Y^{M_0}_\w \in \X^{M_0}_\w$ with $Y^{M_0}_\b \cap Y^{M_0}_\w \neq \emptyset$ and
\begin{equation}   \label{eqPfthmPerturbToStrongNonIntegrable_Y_location}
Y^{M_0}_\b \cup Y^{M_0}_\w \subseteq \inte  ( f  ( \xi_0  )  ) \setminus \xi_0.
\end{equation}
We fix such $Y^{M_0}_\b \in \X^{M_0}_\b$ and $Y^{M_0}_\w \in \X^{M_0}_\w$. See Figure~\ref{figPerturb}.

\smallskip

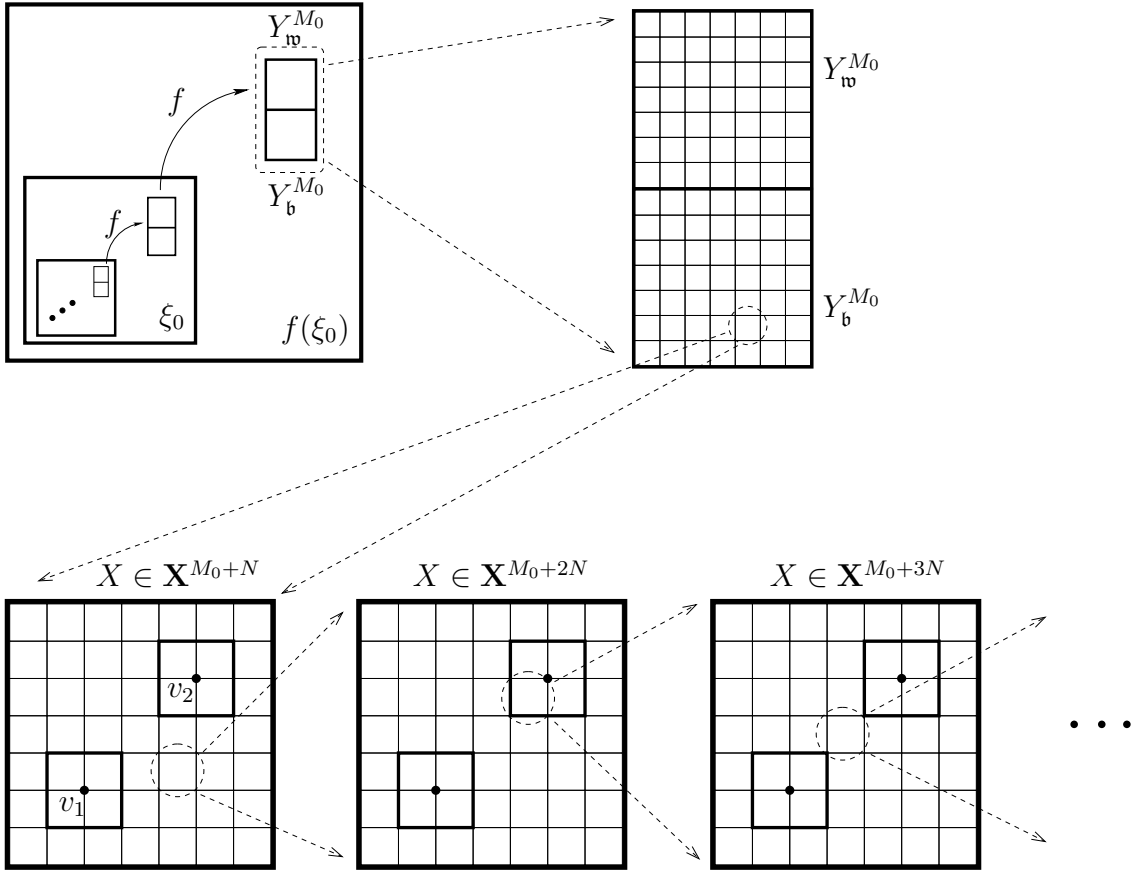
\begin{figure}
    \centering
    \begin{overpic}
    [width=15cm, 
    tics=20]{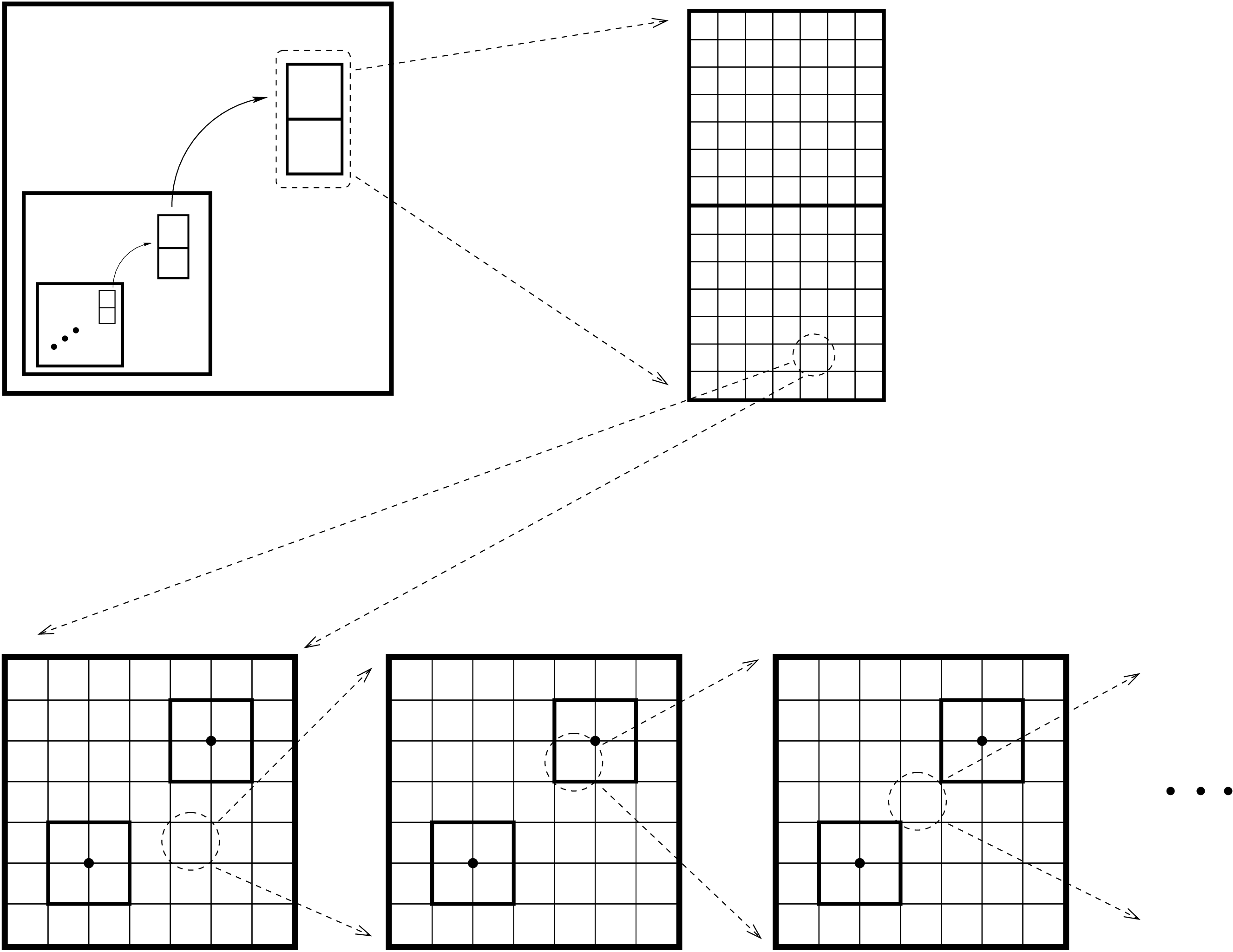}
    \put(100,316){$Y^{M_0}_\w$}
    \put(100,252){$Y^{M_0}_\b$}
    \put(62,288){$f$}
    \put(38,242){$f$}
    \put(105,202){$f(\xi_0)$}
    \put(59,207){$\xi_0$}
    \put(310,300){$Y^{M_0}_\w$}
    \put(310,210){$Y^{M_0}_\b$}
    \put(35,108){$X\in\X^{M_0 + N}$}
    \put(155,108){$X\in\X^{M_0 + 2N}$}
    \put(290,108){$X\in\X^{M_0 + 3N}$}
    \put(21,23){$v_1$}
    \put(62,66){$v_2$}
    \end{overpic}
    \caption{Constructions for the proof of Theorem~\ref{thmPerturbToStrongNonIntegrable}.}
    \label{figPerturb}
\end{figure}

We want to construct, for each $n\in\N_0$ and each $(n+N)$-vertex $v\in \V^{n+N}$, a non-negative bump function $\Upsilon_{v,n} \: S^2 \rightarrow [0,+\infty)$ that satisfies the following properties:
\begin{enumerate}
\smallskip
\item[(a)] $\Upsilon_{v,n} (v) = C_{\dagger}\Lambda^{-\alpha n} \varepsilon$ and $\Upsilon_{v,n} (x) = 0$ if $x\in S^2 \setminus W^{n+N} (v)$.

\smallskip
\item[(b)] $\norm{\Upsilon_{v,n}}_{\CCC^0(S^2)} = C_{\dagger}\Lambda^{-\alpha n} \varepsilon$.

\smallskip
\item[(c)] For each $m\in\N$, each $X\in \X^{n+mN}$, and each pair of points $x, \, y\in X$,
\begin{equation}  \label{eqPfthmPerturbToStrongNonIntegrable_SingleBumpDiffInEachScale}
\abs{ \Upsilon_{v,n} (x) - \Upsilon_{v,n} (y) } \leq C_{\dagger} \Lambda^{-\alpha n} \varepsilon (D_N - 1)^{-(m-1)}.
\end{equation}
\end{enumerate}

\smallskip

Fix arbitrary $n\in\N_0$ and $v\in \V^{n+N}$.

In order to construct such $\Upsilon_{v,n}$, we first need to construct a collection of sets whose boundaries serve as level sets of $\Upsilon_{v,n}$. More precisely, we will construct a collection of closed subsets $\{U_{\underline{i}} \}_{\underline{i} \in I}$ of $W^{n+N}(v)$ indexed by
\begin{equation}  \label{eqPfthmPerturbToStrongNonIntegrable_I}
I \coloneqq \bigcup_{k\in \N} \{0, \, 1, \, \dots,  \, D_N-1 \}^k
\end{equation}
that satisfy the following properties:
\begin{enumerate}
\smallskip
\item[(1)] $U_{\underline{i}}$ is either $\{v\}$ or a nonempty union of $(n+(k+1)N)$-tiles if the \emph{length} of $\underline{i}\in I$ is $k\in\N$, i.e., $\underline{i} \in \{ 0,  \, 1,  \, \dots,  \, D_N - 1 \}^k$. Moreover, $U_{\underline{i}} = \{v\}$ if and only if $\underline{i} \eqqcolon (i_1, i_2, \dots, i_k) = (0, 0 , \dots, 0)$.

\smallskip
\item[(2)] $S^2 \setminus U_{\underline{i}}$ is a finite disjoint union of simply connected open sets for each $\underline{i}\in I$.

\smallskip
\item[(3)] $U_{(i_1,i_2,\dots, i_k)} = U_{(i_1,i_2,\dots, i_k,0)}$ for each $k\in\N$ and each $\underline{i} = (i_1,i_2,\dots, i_k) \in I$.

\smallskip
\item[(4)] $U_{\underline{i}} \subseteq \inter U_{\underline{j}}  \subseteq  U_{\underline{j}}  \subseteq  W^{n+N}(v)$ for all $\underline{i},  \, \underline{j} \in I$ with $\underline{i} < \underline{j}$.
\end{enumerate}
Here we say $\underline{i} < \underline{j}$, for $\underline{i} = (i_1, i_2, \dots, i_k) \in I$ and $\underline{j} = (j_1, j_2, \dots, j_{k'}) \in I$, if one of the following statements holds:
\begin{itemize}
\smallskip
\item $k<k'$, $i_l = j_l$ for all $l \in \N$ with $l\leq k$, and $j_{l'} \neq 0$ for some $l'\in\N$ with $k < l' \leq k'$. 

\smallskip
\item There exists $l'\in\N$ with $l' \leq \min\{k, k'\}$ such that $i_{l'} < j_{l'}$ and $i_l = j_l$ for all $l\in\N$ with $l< l'$.
\end{itemize}

We say $\underline{i} \leq \underline{j}$ for $\underline{i}, \underline{j} \in I$ if either $\underline{i} < \underline{j}$ or $\underline{i} = \underline{j}$.

We denote
\begin{equation}  \label{eqPfthmPerturbToStrongNonIntegrable_Ik}  
I_0 \coloneqq \emptyset, \text{ and } I_l \coloneqq \bigcup_{k=1}^l  \{1,  \, \dots,  \, D_N - 1\}^k \text{ for each } l\in\N.
\end{equation}
\smallskip

We construct $U_{\underline{i}}$ recursively on the length of $\underline{i}\in I$. 

We set $U_{(0)} \coloneqq \{v\}$. For $\underline{i} = (i_1)$, $i_1\in \{1,  \, \dots,  \, D_N - 1\}$, we define a connected closed set 
\begin{align*}
U_{(i_1)}  \coloneqq \bigcup \biggl\{  X_{i_1} & :    \text{there exist } X_1,  \, X_2,  \, \dots,  \, X_{i_1} \in \X^{n+2N} \\
                                                                                           &\quad \text{ such that }  \bigcup_{m=1}^{i_1} X_m \text{ is connected and } v \in   X_1   \biggr\}.
\end{align*}
Note that $U_{(i_1)} \subseteq W^{n+N}(v)$ for $i_1\in \{1,  \, 2,  \, \dots,  \, D_N - 1\}$ since otherwise there would exist $X_1,  \, X_2,  \, \dots,  \, X_{i_1} \in \X^{n+2N}$ such that the union $\bigcup_{m=1}^{i_1} f^{n+N}(X_m)$ of $N$-tiles $f^{n+N}(X_m)\in \X^N$ (see Proposition~\ref{propCellDecomp}~(i)), $m \in \{1, \, 2, \, \dots, \, i_1\}$, is connected and joins opposite sides of $\CC$ which is impossible due to the definition of $D_N$ (see (\ref{eqDefDn})). Then Properties~(1), (2), and (4) hold for $\underline{i}, \, \underline{j} \in \{0,  \, 1,  \, \dots,  \, D_N - 1\}^1$ by our construction.

\begin{figure}
    \centering
    \begin{overpic}
    [width=15cm, 
    tics=20]{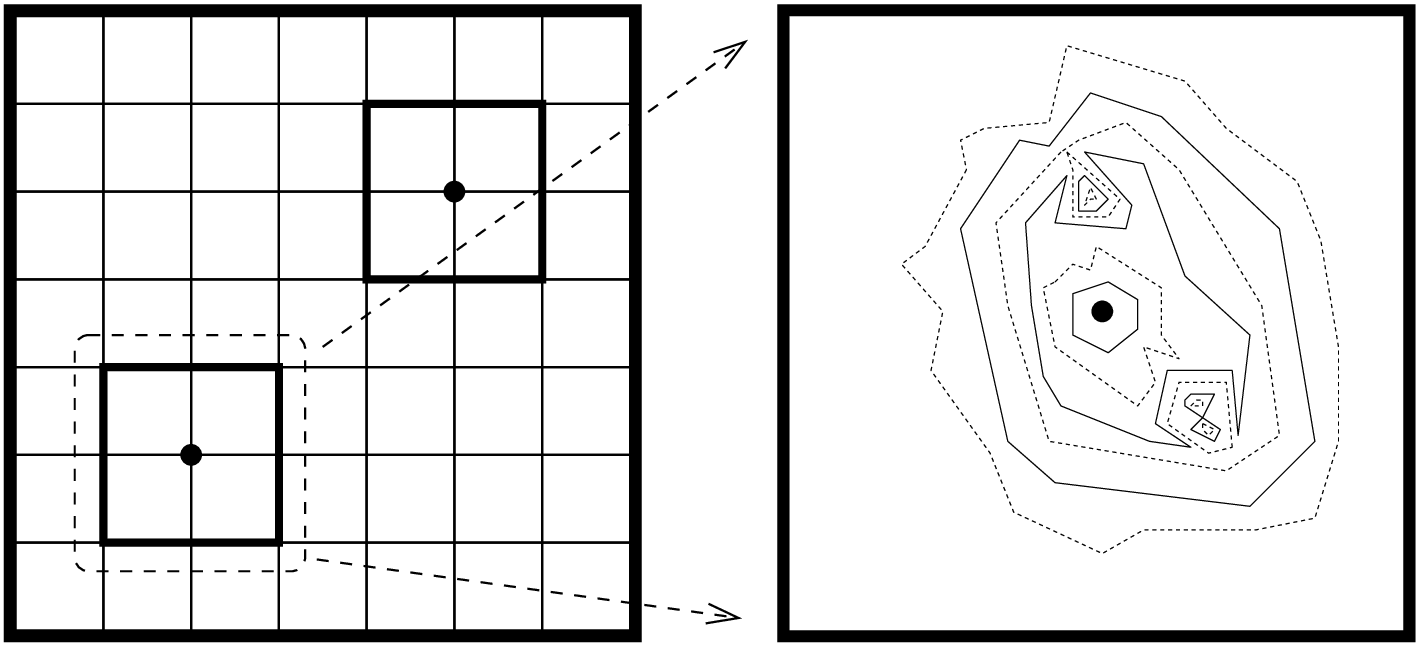}
    \put(47,45){$v_1(X)$}
    \put(126,142){$v_2(X)$}
    \put(5,-10){$X\in\X^{M_0 + N}$}
    \put(239,9){$W^{M_0 + 2 N} (v_1(X))$}
    \end{overpic}
    \caption{Level sets $\partial U_{(i_1)}$, $i_1\in\{1, \, 2, \, \dots,  \, D_N - 1\}$, of $\Upsilon_{v_1(X),\,M_0 + N}$.}
    \label{figBumpFn1}

    \centering
    \begin{overpic}
    [width=15cm, 
    tics=20]{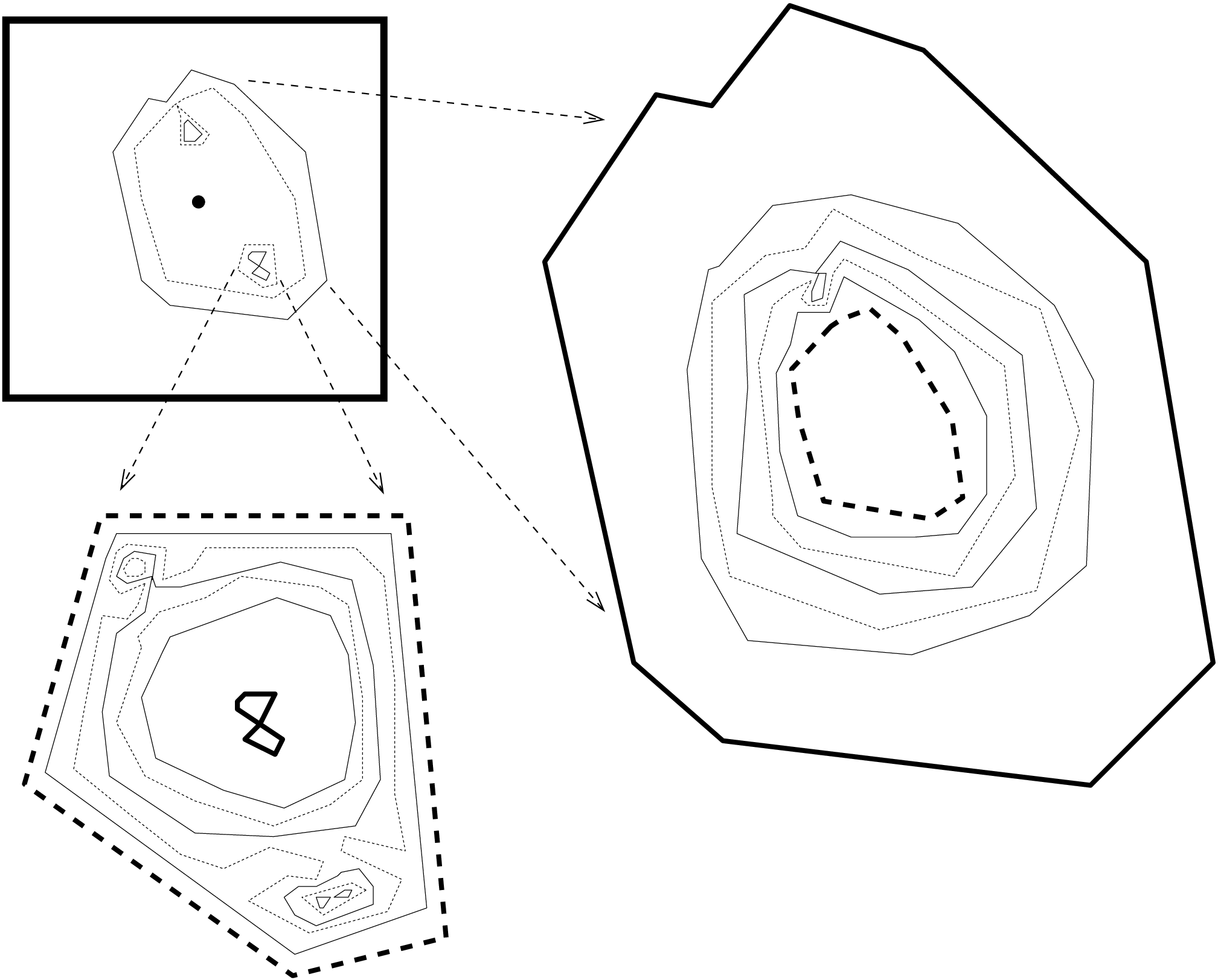}
    \put(60,278){$v_1(X)$}
    \put(5,324){$W^{M_0 + 2 N} (v_1(X))$}
    \end{overpic}
    \caption{Level sets $\partial U_{(4,i_2)}$, $i_2\in\{1, \, 2, \, \dots,  \, D_N - 1\}$, of $\Upsilon_{v_1(X),\,M_0 + N}$.}
    \label{figBumpFn2}
\end{figure}

Assume that we have constructed $U_{\underline{i}} \subseteq W^{n+N}(v)$ for each $\underline{i} \in I_l$ for some $l\in\N$, that Property~(3) holds for each $\underline{i}\in I_{l-1}$, and that Properties~(1), (2), and (4) hold for all $\underline{i},  \, \underline{j} \in I_l$.

Fix arbitrary $\underline{i} = (i_1, i_2, \dots, i_l ) \in  \{0,  \, 1, \,  \dots,  \, D_N - 1\}^l$ and $i_{l + 1} \in \{1,  \, 2,  \, \dots,  \, D_N - 1\}$. Denote $\underline{j} \coloneqq (i_1, i_2, \dots, i_l, i_{l + 1})$. Set $U_{(i_1, i_2, \dots, i_l, 0 )} \coloneqq U_{\underline{i}}$. We define a connected closed set
\begin{align*}
U_{\underline{j}}  \coloneqq U_{\underline{i}}  \cup 
                           \bigcup \biggl\{  X_{i_{l + 1}}  & :    \text{there exist } X_1,  \, X_2,  \, \dots,  \, X_{i_{l + 1}} \in \X^{n+(l + 2)N} \\
                                                                                               &\quad \text{ such that }  \bigcup_{m=1}^{i_{l + 1}} X_m \text{ is connected and } U_{\underline{i}} \cap  X_1 \neq \emptyset   \biggr\}.
\end{align*} 

\smallskip

\emph{Claim~1.} $U_{\underline{j}} \subseteq \inter U_{(i_1, i_2, \dots, i_{l-1}, 1 + i_l)}$ if $i_l \neq D_N - 1$, and $U_{\underline{j}} \subseteq W^{n+N}(v)$ if $i_l = D_N - 1$.

\smallskip

We first establish Claim~1 in the case $i_l \neq D_N - 1$. Denote $\underline{i'} \coloneqq (i_1, i_2, \dots, i_{l-1}, 1 + i_l)$. By Property~(1) of $\{U_{\underline{i}} \}_{\underline{i} \in I_l}$, $U_{\underline{i}}$ and $U_{\underline{i'}}$ are unions of $(n + (l+1)N)$-tiles. By Property~(4) of $\{U_{\underline{i}} \}_{\underline{i} \in I_l}$, $U_{\underline{i}} \subseteq \inter U_{\underline{i'}}$, so $\partial U_{\underline{i}} \cap \partial U_{\underline{i'}} = \emptyset$. We argue by contradiction and assume that $U_{\underline{j}} \nsubseteq \inter U_{\underline{i'}}$. Then there exist $X_1,  \, X_2,  \, \dots,  \, X_{i_{l+1}} \in \X^{n+(l+2)N}$ such that the union $K \coloneqq \bigcup_{m=1}^{i_{l + 1}} X_m$ is a connected set that intersects both $\partial U_{\underline{i}}$ and $\partial U_{\underline{i'}}$ nontrivially. Then $K$ cannot be a subset of a single $(n+(l+1)N)$-flower (of an $(n+(l+1)N)$-vertex).  Since each connected component of the preimage of a $0$-flower under $f^{n+(l+1)N}$ is an $(n+(l+1)N)$-flower, we observe that $f^{n+(l+1)N}(K)$ cannot be a subset of a single $0$-flower (of a $0$-vertex), or equivalently (see \cite[Lemma~5.33]{BM17}), $f^{n+(l+1)N}(K)$ joins opposite sides of $\CC$. Since $f^{n+(l+1)N}(K) = \bigcup_{m=1}^{i_{l + 1}} f^{n+(l+1)N} (X_m)$ is connected, $\bigl\{  f^{n+(l+1)N} (X_m)  :  m\in\{1, \, 2, \, \dots, \, i_{l+1} \}  \bigr\}    \subseteq \X^N$ (see Proposition~\ref{propCellDecomp}~(i))
, and $i_{l+1} \leq D_N - 1$, we get a contradiction to the definition of $D_N$ (see (\ref{eqDefDn})).

Claim~1 is now proved in the case $i_l \neq D_N - 1$. The argument for the proof of the case $i_l = D_N - 1$ is similar, and we omit it here.

\smallskip

By Claim~1 and Property~(4) of $\{U_{\underline{i}} \}_{\underline{i} \in I_l}$, we have  $U_{\underline{j}} \subseteq W^{n+N}(v)$.

Then Properties~(1) and (2) hold for each $\underline{i} \in \{0,  \, 1,  \, \dots,  \, D_N - 1\}^{l+1}$, Property~(3) holds for each $\underline{i} \in \{0,  \, 1,  \, \dots,  \, D_N - 1\}^l$. In order to verify Property~(4) of $\{U_{\underline{i}} \}_{\underline{i} \in I_{l+1}}$, it suffices to observe that by Claim~1 and our construction, for all $\underline{j} \in I_l$ and 
\begin{equation*}
	i_1,  \, i_2,  \, \dots,  \, i_l, i_{l+1},  \, i'_{l+1} \in \{0,  \, 1,  \, \dots,  \, D_N - 1\}
\end{equation*}
with $1 \leq i_{l+1} < i'_{l+1}$ and $\underline{i} \coloneqq (i_1, i_2, \dots, i_l) < \underline{j}$, we have
\begin{equation*}
U_{\underline{i}} \subseteq \inter U_{\underline{i^1}}   \subseteq  U_{\underline{i^1}}  \subseteq \inter U_{\underline{i^2}}  \subseteq U_{\underline{i^2}}  \subseteq \inter U_{\underline{j}},
\end{equation*}
where $\underline{i^1} \coloneqq (i_1, i_2, \dots, i_l, i_{l+1})$ and $\underline{i^2} \coloneqq (i_1, i_2, \dots, i_l, i'_{l+1})$.

The construction of $\{U_{\underline{i}} \}_{\underline{i} \in I}$ and the verification of Properties~(1) through (4) is now complete.

\smallskip

We can now construct the bump function $\Upsilon_{v,\,n} \: S^2 \rightarrow [0, +\infty)$ and verify that it satisfies Properties~(a) through (c) of the bump functions.

We define 
\begin{equation}   \label{eqPfthmPerturbToStrongNonIntegrable_UpsilonDefOnBoundary}
\Upsilon_{v,\,n}(v) \coloneqq C_{\dagger} \Lambda^{-\alpha n} \varepsilon 
\qquad\text{ and } \qquad 
\Upsilon_{v,\,n}(x) \coloneqq 0 \text{ if } x\in S^2 \setminus U_{(D_N-1)}.
\end{equation}
Property~(a) of the bump functions follows from Property~(4) of $\{U_{\underline{i}} \}_{\underline{i} \in I}$.

We denote, for each $k\in\N$,
\begin{equation*}    
I^*_k \coloneqq \{ (i_1, i_2, \dots, i_k) \in I_k  :  i_k \neq 0,  \, i_l \neq D_N -1 \text{ for } 1 \leq l < k \}.
\end{equation*}
Define $I^* \coloneqq \bigcup_{k\in\N}  I^*_k$.

For arbitrary $k\in \N$ and $\underline{i} = (i_1, i_2, \dots, i_k) \in I^*_k$, we define a subset $A_{\underline{i}}$ of $W^{n+N} (v)$ by
\begin{equation} \label{eqPfthmPerturbToStrongNonIntegrable_Annulus}
A_{\underline{i}}  \coloneqq U_{(i_1, i_2, \dots, i_{k-1}, i_k)} \setminus U_{(i_1, i_2, \dots, i_{k-1}, i_k - 1, D_N - 1)}.
\end{equation}
In particular $A_{(i_1)}  = U_{(i_1)} \setminus U_{ (i_1 - 1, D_N - 1) }$ for $i_1 \in \{1, \, 2,  \, \dots,  \, D_N - 1 \}$. We note that by Property~(4) of $\{U_{\underline{i}} \}_{\underline{i} \in I}$,
\begin{equation} \label{eqPfthmPerturbToStrongNonIntegrable_DistinctADisjoint}
A_{\underline{i}}  \cap  A_{\underline{j}}  = \emptyset    \qquad\text{for all } \underline{i}, \underline{j}  \in I^* \text{ with } \underline{i} \neq \underline{j}.
\end{equation}
Thus, we define, for each $k\in \N$ and each $\underline{i} = (i_1, i_2, \dots, i_k) \in I^*_k$,
\begin{equation}  \label{eqPfthmPerturbToStrongNonIntegrable_UpsilonDefOnAnnulus}
\Upsilon_{v, \, n}(x) \coloneqq C_{\dagger} \Lambda^{-\alpha n} \varepsilon \bigg( 1 - \sum_{j=1}^k \frac{ i_j } { (D_N - 1)^j } \biggr) 
\end{equation}
for each $x \in A_{\underline{i}}$.

With abuse of notation, for each $\underline{i}\in I^*$, we write $\Upsilon_{v,\,n}( A_{\underline{i}} ) \coloneqq \Upsilon_{v,\,n}(x)$ for any $x\in A_{\underline{i}}$.

So far we have defined $\Upsilon_{v,\,n}$ on
\begin{equation}   \label{eqPfthmPerturbToStrongNonIntegrable_UnionA}
\mathfrak{U} \coloneqq  \{v\} \cup \bigl(S^2 \setminus U_{(D_N - 1)}  \bigr) \cup \bigcup_{\underline{i} \in I^*}  A_{\underline{i}} .
\end{equation}

\smallskip

\emph{Claim~2.} The set $\mathfrak{U}$ contains all vertices, i.e., $\bigcup_{k\in\N_0} \V^k  \subseteq  \mathfrak{U}$.

\smallskip

In order to establish Claim~2, it suffices to show that $x \in \mathfrak{U}$ for each $x\in \V^{n + (m+1) N} \cap U_{(D_N - 1)}  \setminus \{v\}$ and each $m\in\N$. We fix an arbitrary integer $m\in\N$ and an arbitrary vertex $x\in \V^{n + (m+1) N} \cap U_{(D_N - 1)} \setminus \{v\}$. We choose a sequence $\{i_k\}_{k\in\N}$ in $\{0,  \, 1,  \, \dots,  \, D_N - 2\}$ recursively as follows:

Let $i_1$ be the largest integer in $\{0, \, 1, \,  \dots,  \, D_N - 2\}$ with $x\notin U_{(i_1)}$. Assume that we have chosen $\{ i_k \}_{k=1}^l$ in $\{0, \, 1,  \, \dots,  \, D_N - 2\}$ for some $l\in\N$ with the property that $x \notin U_{(i_1,i_2,\dots,i_l)}$ and $x \in U_{(i_1,i_2,\dots,i_{l-1}, 1 + i_l)}$, then by Properties~(3) and (4) of  $\{U_{\underline{i}} \}_{\underline{i} \in I}$, we can choose $i_{l+1}$ to be the largest integer in $\{0, \, 1,  \, \dots,  \, D_N - 1\}$ with $x\notin U_{(i_1, i_2, \dots, i_{l+1})}$. Assume that $i_{l+1} = D_N - 1$. Thus, $(i_1, i_2, \dots, i_{l-1}, 1 + i_l) \in I^*$ and $x\in U_{(i_1, i_2, \dots, i_{l-1}, 1 + i_l)}  \setminus U_{(i_1, i_2, \dots, i_l, D_N - 1)} = A_{(i_1, i_2, \dots, i_{l-1}, 1 + i_l)}$. 

So we can assume, without loss of generality, that $i_{k} \neq D_N - 1$ for all $k\in\N$, i.e., $\{ i_k:  k\in \N \} \subseteq \{ 0,  \, 1,  \, \dots,  \, D_N - 2\}$ can be constructed above. Then $x \in U_{(i_1, i_2, \dots, i_{m-1}, 1 + i_m)}$. Since both $U_{(i_1, i_2, \dots, i_{m-1}, 1 + i_m)}$ and $U_{(i_1, i_2, \dots,  i_m)}$ are unions of $(n + (m+1) N)$-tiles (see Property~(1) of $\{U_{\underline{i}} \}_{\underline{i} \in I}$), we can see that $x\notin U_{(i_1, i_2, \dots,  i_m, D_N - 1)}$ since otherwise there would exist $X_1,  \, X_2,  \, \dots,  \, X_{D_N - 1} \in \X^{n + (m + 2) N}$ such that the union $K \coloneqq \bigcup_{k=1}^{D_N - 1} X_k$ is connected and have nontrivial intersections with $U_{(i_1, i_2, \dots,i_m)}$ and $\{x\}$, and consequently $K\cap \partial W^{n + (m + 1) N} (x) \neq \emptyset$. This is impossible since $f^{n + (m + 1) N} (K)$, as a union of $N$-tiles $f^{n + (m + 1) N}(X_l)$ (see Proposition~\ref{propCellDecomp}~(i)), $l\in\{1, \, 2, \, \dots,  \, D_N - 1\}$, cannot join opposite sides of $\CC$ due to the definition of $D_N$ in (\ref{eqDefDn}). Hence, $(i_1, i_2, \dots, i_{m-1}, 1 + i_m) \in I^*$ and $x\in U_{(i_1, i_2, \dots, i_{m-1}, 1 + i_m)}  \setminus U_{(i_1, i_2, \dots, i_m, D_N - 1)} = A_{(i_1, i_2, \dots, i_{m-1}, 1 + i_m)}$. Claim~2 is now established.

\smallskip

\emph{Claim~3.} For the function $\Upsilon_{v,\,n}$ defined on $\mathfrak{U}$, inequality~(\ref{eqPfthmPerturbToStrongNonIntegrable_SingleBumpDiffInEachScale}) holds for each $m\in\N$, each $X\in \X^{n+mN}$, and each pair of points $x, \, y\in X\cap \mathfrak{U}$.

\smallskip

Fix arbitrary $m\in\N$, $X\in\X^{n+mN}$, and $x, \, y\in X\cap \mathfrak{U}$. Inequality~(\ref{eqPfthmPerturbToStrongNonIntegrable_SingleBumpDiffInEachScale}) holds for $x, \, y \in X \cap \mathfrak{U}$ trivially if $m=1$ by (\ref{eqPfthmPerturbToStrongNonIntegrable_UpsilonDefOnBoundary}) and (\ref{eqPfthmPerturbToStrongNonIntegrable_UpsilonDefOnAnnulus}). So without loss of generality, we can assume $m\geq 2$. We choose a sequence $\{i_k\}_{k\in\N}$ in $\{0,  \, 1,  \, \dots,  \, D_N - 1\}$ recursively as follows:

Let $i_1$ be the largest integer in $\{0, \, 1,  \, \dots,  \, D_N - 1\}$ with $X \nsubseteq U_{(i_1)}$. Assume that we have chosen $\{ i_k \}_{k=1}^l$ for some $l\in\N$ with the property that $X \nsubseteq U_{(i_1,i_2,\dots,i_l)}$, then by Properties~(3) and (4) of  $\{U_{\underline{i}} \}_{\underline{i} \in I}$, we can choose $i_{l+1}$ to be the largest integer in $\{0, \, 1,  \, \dots,  \, D_N - 1\}$ with $X \nsubseteq U_{(i_1, i_2, \dots, i_{l+1})}$.

We establish Claim~3 by considering the following two cases:

\smallskip

\emph{Case 1.} $i_k = D_N - 1$ for some integer $k\in [1, m-1]$. Without loss of generality, we assume that $k$ is the smallest such integer. Recall that $m\geq 2$. If $k=1$, then by Property~(1) of $\{U_{\underline{i}} \}_{\underline{i} \in I}$, $X\subseteq \bigl(S^2 \setminus \inter U_{(D_N - 1)} \bigr) \subseteq \bigl(S^2 \setminus  U_{(D_N - 1)} \bigr) \cup A_{(D_N - 1)}$, and consequently $\Upsilon(x) = 0 = \Upsilon(y)$ by (\ref{eqPfthmPerturbToStrongNonIntegrable_UpsilonDefOnBoundary}) and (\ref{eqPfthmPerturbToStrongNonIntegrable_UpsilonDefOnAnnulus}). If $k \geq 2$, then $(i_1, i_2, \dots, i_{k-2}, 1 + i_{k-1} ),  \, (i_1, i_2, \dots, i_{k-1}, D_N - 1 )  \in I^*$, and
\begin{equation*}
X   \subseteq U_{(i_1, i_2, \dots, i_{k-2}, 1 + i_{k-1} )} \setminus \inter U_{(i_1, i_2, \dots, i_{k-1}, D_N - 1 )}  
      \subseteq A_{(i_1, i_2, \dots, i_{k-2}, 1 + i_{k-1} )} \cup A_{(i_1, i_2, \dots, i_{k-1}, D_N - 1 )}
\end{equation*}
by our choice of $i_{k-1}$, the fact that both $U_{(i_1, i_2, \dots, i_{k-2}, 1 + i_{k-1} )}$ and $U_{(i_1, i_2, \dots, i_{k-1}, D_N - 1 )}$ are unions of $(n + (k+1) N)$-tiles (by Property~(1) of $\{U_{\underline{i}} \}_{\underline{i} \in I}$), and (\ref{eqPfthmPerturbToStrongNonIntegrable_Annulus}).  Hence, by (\ref{eqPfthmPerturbToStrongNonIntegrable_UpsilonDefOnAnnulus}), $\Upsilon_{v, \, n}(x) = C_{\dagger} \Lambda^{-\alpha n} \varepsilon \bigl( 1 - \sum_{j=1}^k \frac{ i_j } { (D_N - 1)^j } \bigr) = \Upsilon_{v, \, n}(y)$.

\smallskip

\emph{Case 2.} $i_k \leq D_N - 2$ for all integer $k\in [1, m-1]$. Then by our choice of $i_{m-1}$ and Properties~(1) and (4) of $\{U_{\underline{i}} \}_{\underline{i} \in I}$,
\begin{equation}   \label{eqPfthmPerturbToStrongNonIntegrable_XinAnnulus}
X \subseteq U_{(i_1, i_2, \dots, i_{m-2}, 1 + i_{m-1})}  \setminus \inter U_{(i_1, i_2, \dots,  i_{m-1})}  \subseteq U_{(i_1, i_2, \dots, i_{m-2}, 1 + i_{m-1})}  \setminus U_{\underline{j}}
\end{equation}
for each $\underline{j} \in I$ with $\underline{j} < (i_1, i_2, \dots,  i_{m-1})$.

Note that by (\ref{eqPfthmPerturbToStrongNonIntegrable_Annulus}) and Property~(4) of $\{U_{\underline{i}} \}_{\underline{i} \in I}$,
\begin{equation}  \label{eqPfthmPerturbToStrongNonIntegrable_AnnulusInU}
A_{\underline{i}} \subseteq U_{\underline{j}}  \text{ for all } \underline{i} \in I^* \text{ and } \underline{j}\in I \text{ with } \underline{i} \leq \underline{j}.
\end{equation}
By (\ref{eqPfthmPerturbToStrongNonIntegrable_UpsilonDefOnBoundary}) and (\ref{eqPfthmPerturbToStrongNonIntegrable_UpsilonDefOnAnnulus}),
\begin{equation}  \label{eqPfthmPerturbToStrongNonIntegrable_UpsilonMonotonicity}
\Upsilon_{v,\,n}(A_{\underline{i}})  \geq \Upsilon_{v,\,n} \bigl( A_{\underline{j}} \bigr)  \text{ for all } \underline{i},  \, \underline{j} \in I^* \text{ with } \underline{i} \leq \underline{j}.
\end{equation}

Thus, by (\ref{eqPfthmPerturbToStrongNonIntegrable_XinAnnulus}), (\ref{eqPfthmPerturbToStrongNonIntegrable_AnnulusInU}), and (\ref{eqPfthmPerturbToStrongNonIntegrable_UpsilonMonotonicity}),
\begin{align*}
              \abs{ \Upsilon_{v,\,n} (x) - \Upsilon_{v,\,n} (y) } 
&\leq     \inf  \bigl\{  \Upsilon_{v,\,n} ( A_{\underline{i}} )    :    \underline{j} \in I, \, \underline{i} \in I^*,  \, \underline{i} \leq \underline{j} < (i_1, i_2, \dots, i_{m-1}) \bigr\}    \\
         &\quad  - \inf  \{  \Upsilon_{v,\,n} ( A_{\underline{i}} )    :    \underline{i} \in I^*,  \, \underline{i} \leq  (i_1, i_2, \dots, i_{m-2}, 1 + i_{m-1})  \}   \\
&\leq   C_{\dagger} \Lambda^{-\alpha n} \varepsilon (D_N - 1)^{-(m-1)},
\end{align*}
where the last identity follows easily from (\ref{eqPfthmPerturbToStrongNonIntegrable_UpsilonDefOnAnnulus}) and the definition of $I^*$ by separate explicit calculations depending on $i_{m-1} = 0$ or not.

Claim~3 is now established.

\smallskip

\emph{Claim~4.} The function $\Upsilon_{v,\,n}$ is continuous on $\mathfrak{U}$.

\smallskip

Fix arbitrary $x, \, y\in \mathfrak{U}$ and $m\in\N$ with $x\neq y$ and $y\in U^{n + m N}(x)$ (see (\ref{defU^n})). Then there exist $X_1,\, X_2 \in \X^{n + m N}$ such that $x\in X_1$, $y\in X_2$, and $X_1 \cap X_2 \neq \emptyset$. It follows immediately from Definition~\ref{defcelldecomp}~(iii) that there exists an $( n + m N )$-vertex $z$ in $X_1 \cap X_2$. Then by Claim~2 and Claim~3, 
\begin{align*}
             \abs{ \Upsilon_{v,\,n}(x)  -  \Upsilon_{v,\,n}(y) } 
&\leq   \abs{ \Upsilon_{v,\,n}(x)  -  \Upsilon_{v,\,n}(z) }   +  \abs{ \Upsilon_{v,\,n}(z)  -  \Upsilon_{v,\,n}(y) }   \\
&\leq   2 C_{\dagger} \Lambda^{-\alpha n} \varepsilon (D_N - 1)^{-(m-1)}. 
\end{align*}
Hence, Claim~4 follows from Lemma~\ref{lmCellBoundsBM}~(iv) and the fact that $D_N - 1 > 1$.

\smallskip

Since we have defined $\Upsilon_{v,n}$ continuously on a dense subset $\mathfrak{U}$ of $S^2$ by Claim~2 and Claim~4, we can now extend $\Upsilon_{v,n}$ continuously to $S^2$. Property~(b) of the bump functions follows immediately from (\ref{eqPfthmPerturbToStrongNonIntegrable_UpsilonDefOnBoundary}) and (\ref{eqPfthmPerturbToStrongNonIntegrable_UpsilonDefOnAnnulus}). Property~(c) of the bump functions follows from Claim~3.

\smallskip

Recall $u^1_\b,  \, u^2_\b,  \, u^1_\w,  \, u^2_\w \in \V^N$ defined above.

For each $n\in\N_0$, each $n$-tile $X\in\X^n$, and each $i\in\{1, \, 2\}$, we define a point
\begin{equation}   \label{eqPfthmPerturbToStrongNonIntegrable_vi}
v_i(X) \coloneqq  \begin{cases} (f^n|_X)^{-1} \bigl( u^i_\b \bigr) & \text{if } X\in\X^n_\b, \\ (f^n|_X)^{-1}  \bigl( u^i_\w  \bigr)   & \text{if } X\in\X^n_\w.  \end{cases}
\end{equation}

Fix an arbitrary real-valued H\"{o}lder continuous function $\varphi \in \Holder{\alpha}(S^2,d)$ with an exponent $\alpha$. 

We are going to construct $\phi \in \Holder{\alpha}(S^2,d)$ for the given $\varphi$ by defining their difference $\Upsilon \in \Holder{\alpha}(S^2,d)$ supported on the (disjoint) backward orbits of $Y^{M_0}_\b \cup Y^{M_0}_\w$ along $\{ \xi_{\minus i} \}_{i\in\N_0}$ as the sum of a collection of non-negative bump functions constructed above.

We construct $\varphi_m \in  \Holder{\alpha}(S^2,d)$ recursively on $m\in\N_0$.

Set $\varphi_0 \coloneqq \varphi$.

Assume that $\varphi_i \in \Holder{\alpha}(S^2,d)$ has been constructed for some $i\in\N_0$, we define a number $\delta_X\in\{0, \, 1\}$, for each $X\in\X^{M_0 + (i+1) N}$ with $X\subseteq Y^{M_0}_\b \cup Y^{M_0}_\w$, by
\begin{equation} \label{eqPfthmPerturbToStrongNonIntegrable_deltaX}
\delta_X \coloneqq  \begin{cases} 1 & \text{if } \Absbig{ (\varphi_i)^{f,\,\CC}_{\xi,\,\xi'}   ( v_1(X), v_2(X) ) } < 2 \varepsilon d( v_1(X), v_2(X) )^\alpha,  \\ 0    &  \text{otherwise}.  \end{cases}
\end{equation} 
We define
\begin{equation}  \label{eqPfthmPerturbToStrongNonIntegrable_varphi_i}
\varphi_{i+1} \coloneqq \varphi_i  +  \sum_{j\in\N}        \sum\limits_{\substack{X \in\X^{M_0 + (i+1) N} \\X \subseteq Y^{M_0}_\b \cup Y^{M_0}_\w}}  
                                                                      \delta_X \Upsilon_{v_1 ( \tau_j(X) ),  \, M_0 + (i+1) N + j}  ,
\end{equation}
and finally define the non-negative bump function $\Upsilon\: S^2 \rightarrow [0,1)$ by
\begin{equation}  \label{eqPfthmPerturbToStrongNonIntegrable_Upsilon}
\Upsilon   \coloneqq   \sum_{j\in\N}    \sum_{m\in\N}     \sum\limits_{\substack{X \in\X^{M_0 + m N} \\X \subseteq Y^{M_0}_\b \cup Y^{M_0}_\w}}  
                                                                      \delta_X \Upsilon_{v_1 ( \tau_j(X) ),  \, M_0 + m N + j}  .
\end{equation}
Here the function $\tau_j$ is defined in (\ref{eqPfthmPerturbToStrongNonIntegrable_tau}). It follows immediately from Property~(b) of the bump functions that the series in (\ref{eqPfthmPerturbToStrongNonIntegrable_varphi_i}) and (\ref{eqPfthmPerturbToStrongNonIntegrable_Upsilon}) converge uniformly and absolutely.

We set $\phi \coloneqq \varphi + \Upsilon$.

For each $\c\in\{\b, \, \w\}$, each integer $M\geq M_0$, and each $M$-tile $X\in \X^M$ with $X\subseteq Y^{M_0}_\c$, we choose an arbitrary $\bigl(M_0 + \bigl\lceil \frac{ M - M_0 }{N} \bigr\rceil N \bigr)$-tile $X'$ with $X' \subseteq X$ and define $x_i(X) \coloneqq v_i(X')$ for each $i\in\{1, \, 2\}$.

Now we discuss some properties of the supports of the terms in the series defining $\Upsilon$ in (\ref{eqPfthmPerturbToStrongNonIntegrable_Upsilon}).  See Figure~\ref{figPerturb}.

Fix arbitrary integers $m, \, j\in\N$, by Property~(a) of the bump functions, (\ref{eqPfthmPerturbToStrongNonIntegrable_vi}), and properties of $u^1_\b,  \, u^1_\w \in \V^N$, we have
\begin{align}   
                        \supp \Upsilon_{ v_1 ( \tau_j(X) ), \, M_0 + m N + j }   
&\subseteq   \overline{W}^{M_0 + (m+1) N + j} \bigl( v_1 \bigl( \tau_j(X) \bigr) \bigr)    \label{eqPfthmPerturbToStrongNonIntegrable_SuppInTile} \\
&\subseteq  \inte \bigl( \tau_j(X) \bigr)
\subseteq \tau_j \bigl( Y^{M_0}_\b \cup Y^{M_0}_\w  \bigr),   \notag
\end{align}
for each $(M_0 + m N)$-tile $X\in \X^{M_0 + m N}$ with $X\subseteq Y^{M_0}_\b \cup Y^{M_0}_\w$. Consequently, by (\ref{eqPfthmPerturbToStrongNonIntegrable_SuppInTile}) and the fact that $\tau_{j_1} \bigl( Y^{M_0}_\b \cup Y^{M_0}_\w \bigr)$ and $\tau_{j_2} \bigl( Y^{M_0}_\b \cup Y^{M_0}_\w \bigr)$ are disjoint for distinct $j_1, \, j_2\in\N$ (see Figure~\ref{figPerturb}), we have
\begin{equation}  \label{eqPfthmPerturbToStrongNonIntegrable_SuppDisjoint}
 \supp \Upsilon_{ v_1 ( \tau_{j_1} (X_1) ), \, M_0 + m N + j_1 }  \cap  \supp \Upsilon_{ v_1 ( \tau_{j_2} (X_2) ), \, M_0 + m N + j_2 }  = \emptyset
\end{equation}
for each pair of integers $j_1,  \, j_2 \in \N$ and each pair of  $(M_0 + m N)$-tiles $X_1,  \, X_2 \in\X^{M_0 + m N}$ with $X_1 \cup X_2 \subseteq Y^{M_0}_{\b} \cup Y^{M_0}_{\w}$ and $(j_1, X_1) \neq (j_2, X_2)$.

We are now ready to verify Property~(iii) in Theorem~\ref{thmPerturbToStrongNonIntegrable}.

\smallskip

\emph{Property~(iii).} By (\ref{eqPfthmPerturbToStrongNonIntegrable_SuppDisjoint}), Property~(b) of the bump functions, and (\ref{eqPfthmPerturbToStrongNonIntegrable_M0}),
\begin{align*}
&                \norm{\Upsilon}_{\CCC^0(S^2)}  \\
&\qquad\leq     \sum_{ m\in\N}  \sup \bigl\{  \norm{   \Upsilon_{ v_1 ( \tau_j(X) ), \, M_0 + m N + j }   }_{\CCC^0(S^2)}   :  j\in\N, \, X\in \X^{M_0 + m N}, \, X\subseteq   Y^{M_0}_\b \cup Y^{M_0}_\w  \bigr\}  \\
&\qquad\leq     \sum_{ m\in\N}  C_{\dagger} \Lambda ^{ - \alpha (M_0 + m N)}  \varepsilon\\
&\qquad\leq        C_{\dagger}   \Lambda^{-\alpha M_0} \varepsilon \big/ \bigl( 1-\Lambda^{-\alpha N} \bigr) \\
&\qquad\leq        \varepsilon /2.
\end{align*}

Fix $x, \, y\in S^2$ with $x\neq y$. 

Note that $\supp \Upsilon \subseteq \bigcup_{j\in\N} \tau_j \bigl( Y^{M_0}_\b \cup Y^{M_0}_\w \bigr)$ and that this union is a disjoint union. We bound $\frac{ \abs{\Upsilon(x) - \Upsilon(y)} } { d(x,y)^\alpha }$ by considering the following cases:

\smallskip

\emph{Case 1.} $x \notin \supp \Upsilon$ and $y\notin \supp \Upsilon$. Then $\Upsilon(x) - \Upsilon(y)=0$.

\smallskip

\emph{Case 2.} $\{x,  \, y\} \cap \tau_j \bigl( Y^{M_0}_\b \cup Y^{M_0}_\w \bigr) \neq \emptyset$ and $\{x,  \, y\} \nsubseteq \tau_j  ( f(\xi_0) \setminus \xi_0  )$ for some $j\in\N$. Without loss of generality, we can assume that $j$ is the smallest such integer. Then by (\ref{eqPfthmPerturbToStrongNonIntegrable_Y_location}), Lemma~\ref{lmCellBoundsBM}~(i), and Property~(b) of the bump functions,
\begin{align*}
&                  \abs{\Upsilon(x) - \Upsilon(y)} / d(x,y)^\alpha  \\
&\qquad  \leq \frac{  \sum_{ m\in\N}  \sup \bigl\{  \norm{   \Upsilon_{ v_1 ( \tau_j(X) ), \, M_0 + m N + j }   }_{\CCC^0(S^2)}   :  X\in \X^{M_0 + m N}, \, X\subseteq   Y^{M_0}_\b \cup Y^{M_0}_\w \bigr\}  }  
                                                            { C^{-\alpha} \Lambda^{-\alpha (M_0 + j) } }  \\
&\qquad  \leq  C^{\alpha} \Lambda^{\alpha (M_0 + j) }      \sum_{ m\in\N}  C_{\dagger} \Lambda ^{ - \alpha (M_0 + m N + j)}  \varepsilon      \\
&\qquad\leq \varepsilon  C  C_{\dagger} \Lambda^{ - \alpha N }  \big/ \bigl( 1 -  \Lambda^{ - \alpha N }   \bigr)     \\
&\qquad\leq \varepsilon    C   C_{\dagger} (3C C_{\dagger})^{-1}  \bigl( 1 -  3^{-1}   \bigr)^{-1}        \\
&\qquad =     \varepsilon / 2 .
\end{align*} 
The last inequality follows from our choice of $N$ at the beginning of this proof.

\smallskip

\emph{Case 3.} $\{x,  \, y\} \cap \tau_j \bigl( Y^{M_0}_\b \cup Y^{M_0}_\w \bigr) \neq \emptyset$ and $\{x,  \, y\} \subseteq \tau_j  ( f(\xi_0) \setminus \xi_0  )$ for some $j\in\N$.  Note that such $j$ is unique. Then by (\ref{eqPfthmPerturbToStrongNonIntegrable_Upsilon}) and our constructions of $Y^{M_0}_\b,  \, Y^{M_0}_\w \in \X^{M_0}$ and $\xi \in \Sigma_{f,\,\CC}^-$, we get that for each $z\in\{x, \, y\}$,
\begin{equation}  \label{eqPfthmPerturbToStrongNonIntegrable_PropertyIIICase3}
\Upsilon (z)  =    \sum_{m\in\N}     \sum\limits_{\substack{X \in\X^{M_0 + m N} \\X \subseteq Y^{M_0}_\b \cup Y^{M_0}_\w}}  
                                                                      \delta_X \Upsilon_{v_1 ( \tau_j(X) ),  \, M_0 + m N + j}  (z).
\end{equation}

Since $f$ is an expanding Thurston map, we can define an integer
\begin{align*}
m_1 \coloneqq \max \bigl\{  k \in \Z    :  \text{there exist } &  X_1,\, X_2 \in \X^{M_0 + k N + j}  \text{ such that } \\
                                                                                        &x\in X_1, \, y\in X_2, \text{ and } X_1 \cap X_2 \neq \emptyset   \bigr\}.
\end{align*}

If $m_1 \leq 0$, then by (\ref{eqPfthmPerturbToStrongNonIntegrable_PropertyIIICase3}), (\ref{eqPfthmPerturbToStrongNonIntegrable_SuppInTile}), Property~(b) of the bump functions, Lemma~\ref{lmCellBoundsBM}~(i), and (\ref{eqDefC27}), we have
\begin{align*}
&                        \abs{\Upsilon(x) - \Upsilon(y)} / d(x,y)^\alpha \\
&\qquad  \leq \sum_{m\in\N}  \frac{ \sup \bigl\{  \norm{   \Upsilon_{ v_1 ( \tau_j(X) ), \, M_0 + m N + j }   }_{\CCC^0(S^2)}   :  X\in \X^{M_0 + m N}, \, X\subseteq   Y^{M_0}_\b \cup Y^{M_0}_\w \bigr\}    } 
                                                                          { d(x,y)^\alpha }                    \\
&\qquad  \leq \bigl( C^{-1}  \Lambda^{ -  (M_0 + N + j )}  \bigr)^{-\alpha} \sum_{m\in\N}  C_{\dagger} \Lambda^{ - \alpha (M_0 + m N + j) } \varepsilon   \\      
&\qquad  \leq   C_{\dagger}   C \bigl(  1 - \Lambda^{ - \alpha N}   \bigr)^{-1}  \varepsilon     \\
&\qquad  \leq    (C_{\sharp} - 1 ) \varepsilon.                                                                 
\end{align*}

If $m_1 \geq 1$, then $y\in U^{M_0 + m_1 N + j} (x)$ and $y \notin U^{M_0 + (m_1 + 1) N + j} (x)$ (see (\ref{defU^n})). Choose $X_1, \, X_2 \in \X^{M_0 + m_1 N + j}$ such that $x\in X_1$, $y\in X_2$, and $X_1 \cap X_2 = \emptyset$. For each $i\in\{1, \, 2\}$ and each $m\in\N$ with $1\leq m \leq m_1$, we denote the unique $(M_0 + m N + j)$-tile containing $X_i$ by $Y^i_m$. Then by (\ref{eqPfthmPerturbToStrongNonIntegrable_PropertyIIICase3}), (\ref{eqPfthmPerturbToStrongNonIntegrable_SuppInTile}), Properties~(b) and (c) of the bump functions, Lemma~\ref{lmCellBoundsBM}~(i), (\ref{eqPfthmPerturbToStrongNonIntegrable_rho}), and (\ref{eqDefC27}),

\begin{align*}
&                       \abs{\Upsilon(x) - \Upsilon(y)} / d(x,y)^\alpha   \\
&\qquad  \leq  \sum_{m\in\N}    \sum\limits_{\substack{X \in\X^{M_0 + m N} \\X \subseteq Y^{M_0}_\b \cup Y^{M_0}_\w}}  
                           \frac{ \delta_X  \abs{ \Upsilon_{v_1 ( \tau_j(X) ),  \, M_0 + m N + j} (x) -  \Upsilon_{v_1 ( \tau_j(X) ),  \, M_0 + m N + j} (y) } } { d(x,y)^\alpha }                                                  \\
&\qquad  \leq    \sum_{m=m_1}^{+\infty} 
                                 \frac{ \sup \bigl\{  \norm{   \Upsilon_{ v_1 ( \tau_j(X) ), \, M_0 + m N + j }   }_{\CCC^0(S^2)}   :  X\in \X^{M_0 + m N}, \, X\subseteq   Y^{M_0}_\b \cup Y^{M_0}_\w \bigr\} } { d(x,y)^\alpha }   \\
&\qquad\qquad + \sum_{m=1}^{m_1 - 1}      \sum_{i\in\{1, \, 2\}}
                                \frac{   \abs{ \Upsilon_{v_1 ( Y^i_m ),  \, M_0 + m N + j} (x) -  \Upsilon_{v_1 ( Y^i_m ),  \, M_0 + m N + j} (y) }  } { d(x,y)^\alpha }                                         \\
&\qquad  \leq \frac{  \sum_{m=m_1}^{+\infty}  C_{\dagger} \Lambda^{ - \alpha (M_0 + m N + j) } \varepsilon 
                                   +  \sum_{m=1}^{m_1 - 1}   4 C_{\dagger} \Lambda^{ - \alpha (M_0 + m N + j) } \varepsilon (D_N - 1)^{ - (m_1 - m - 1)  }  } 
                                   {  C^{-\alpha}  \Lambda^{ - \alpha (M_0 + (m_1 + 1) N + j )}  }   \\
&\qquad  \leq   C_{\dagger}   C \bigl( \Lambda^{ \alpha N} \bigl( 1 - \Lambda^{ - \alpha N} \bigr)^{-1} + 4   (1-\rho)^{-1} \bigr)  \varepsilon     \\
&\qquad       =    (C_{\sharp} - 1 ) \varepsilon.                                                                        
\end{align*}

\smallskip

To summarize, we have shown that 
\begin{equation*}
	\Hnorm{\alpha}{\phi - \varphi}{(S^2,d)} \leq \biggl(\frac{1}{2} +\frac{1}{2} + C_{\sharp} - 1 \biggr) \varepsilon  = C_{\sharp} \varepsilon, 
\end{equation*}
establishing Property~(iii) in Theorem~\ref{thmPerturbToStrongNonIntegrable}.

\smallskip

Finally, we are going to verify Properties~(i) and (ii) in Theorem~\ref{thmPerturbToStrongNonIntegrable}.

Fix arbitrary $\c\in\{\b, \, \w\}$, $M\in\N$ with $M\geq M_0$, and $X_0 \in \X^M$ with $X_0 \subseteq Y^{M_0}_\c$. Denote $m_0 \coloneqq \bigl\lceil \frac{ M - M_0 }{N} \bigr\rceil$, $M' \coloneqq M_0 + m_0 N \in [M, M+N)$, and fix $X' \in \X^{M'}$ with $x_1(X_0) = v_1(X') \in \V^{M' + N}$ and $x_2(X_0) = v_2(X') \in \V^{M' + N}$.

\smallskip

\emph{Property~(i).} Fix arbitrary $i\in\{1, \, 2\}$. Since $\overline{W}^{M' + N} (x_i(X_0)) \subseteq \inte(X') \subseteq \inte(X_0)$ and $\overline{W}^{M' + N} (x_1(X_0)) \cap \overline{W}^{M' + N} (x_2(X_0)) = \emptyset$ (which follows from (\ref{eqPfthmPerturbToStrongNonIntegrable_Flower})), we get from Lemma~\ref{lmCellBoundsBM}~(i) and (ii) that
\begin{equation*}
d \bigl( x_i(X_0), S^2\setminus X_0 \bigr) 
\geq C^{-1} \Lambda^{ - (M'+N) } 
\geq C^{-1} \Lambda^{ - M - 2N } 
\geq C^{-2} \Lambda^{ - 2N }  \diam_d( X_0 ),
\end{equation*}
and similarly,
\begin{equation*}
d  ( x_1(X_0),  x_2(X_0)  ) 
\geq C^{-1} \Lambda^{ - (M'+N) } 
\geq C^{-1} \Lambda^{ - M - 2N } 
\geq C^{-2} \Lambda^{ - 2N }  \diam_d( X_0 ).
\end{equation*}

Property~(i) in Theorem~\ref{thmPerturbToStrongNonIntegrable} now follows from (\ref{eqPfthmPerturbToStrongNonIntegrable_varepsilon}).

\smallskip

\emph{Property~(ii).} We first show
\begin{equation}   \label{eqPfthmPerturbToStrongNonIntegrable_PrePropertyII}
 \Absbig{ \phi^{f,\,\CC}_{\xi,\,\xi'} ( x_1(X_0), x_2(X_0) ) } \geq   2 \varepsilon d( x_1(X_0), x_2(X_0) )^\alpha.
\end{equation}

Indeed, observe that by our construction and (\ref{eqPfthmPerturbToStrongNonIntegrable_SuppInTile}), for each integer $m>m_0$, the sets
\begin{equation*}
                     \bigcup_{j\in\N}     \bigcup\limits_{\substack{X \in\X^{M_0 + m N} \\X \subseteq Y^{M_0}_\b \cup Y^{M_0}_\w}}  \supp \Upsilon_{ v_1 ( \tau_j(X) ), \, M_0 + m N + j }
\subseteq   \bigcup_{j\in\N}     \bigcup\limits_{\substack{X \in\X^{M_0 + m N} \\X \subseteq Y^{M_0}_\b \cup Y^{M_0}_\w}}    \inte( \tau_j(X) )
\end{equation*}
are disjoint from the backward orbits of $v_1(X') \in \V^{M_0 + (m_0 +1) N}$ and $v_2(X') \in \V^{M_0 + (m_0 +1) N}$ under $\xi$ and $\xi'$. Thus, by (\ref{eqPfthmPerturbToStrongNonIntegrable_varphi_i}),
\begin{align*}
&                   \Absbig{ \phi^{f,\,\CC}_{\xi,\,\xi'} ( x_1(X_0), x_2(X_0) ) } \\
&\qquad =   \Absbig{ \phi^{f,\,\CC}_{\xi,\,\xi'} ( v_1(X'),    v_2(X') ) }    \\
&\qquad =   \Absbigg{ \Bigl( \varphi_{m_0} +  \sum_{j\in\N}    \sum_{m=m_0 + 1}^{+\infty}     \sum\limits_{\substack{X \in\X^{M_0 + m N} \\X \subseteq Y^{M_0}_\b \cup Y^{M_0}_\w}}  
                                                                      \delta_X \Upsilon_{v_1 ( \tau_j(X) ),  \, M_0 + m N + j}  \Bigr)^{f,\,\CC}_{\xi,\,\xi'} ( v_1(X'), v_2(X') ) }  \\
&\qquad =   \Absbig{ ( \varphi_{m_0} )^{f,\,\CC}_{\xi,\,\xi'} ( v_1(X'),    v_2(X') ) }                                                                  .
\end{align*}

We observe that for each $j\in\N$, the sets
\begin{equation*}
                   \bigcup_{j\in\N}    \bigcup\limits_{\substack{X \in\X^{M_0 + m_0 N}  \setminus \{ X' \} \\ X \subseteq  Y^{M_0}_\b \cup Y^{M_0}_\w   }}  \supp \Upsilon_{ v_1 ( \tau_j(X) ), \, M_0 + m_0 N + j }
\subseteq  \bigcup_{j\in\N}    \bigcup\limits_{\substack{X \in\X^{M_0 + m_0 N}  \setminus \{ X' \} \\ X \subseteq  Y^{M_0}_\b \cup Y^{M_0}_\w   }}   \inte( \tau_j(X) )
\end{equation*}
are disjoint from the backward orbits of $v_1(X')$ and $v_2(X')$ under $\xi$ and $\xi'$ (by (\ref{eqPfthmPerturbToStrongNonIntegrable_SuppInTile}) and our choices of $\xi$ and $\xi'$ from Lemma~\ref{lmDisjointBackwardOrbits}). See Figure~\ref{figPerturb}. Thus, for each $X\in \X^{M_0 + m_0 N}$ with $X \subseteq  Y^{M_0}_\b \cup Y^{M_0}_\w$ and $X \neq X'$,  we have
\begin{equation}    \label{eqPfthmPerturbToStrongNonIntegrable_UpsilonTempDistr0}
( \Upsilon_{v_1 ( \tau_j(X) ),  \, M_0 + m_0 N + j}  )^{f,\,\CC}_{\xi,\,\xi'} ( v_1(X'), v_2(X') )  = 0.
\end{equation}

By our construction in (\ref{eqPfthmPerturbToStrongNonIntegrable_deltaX}) and (\ref{eqPfthmPerturbToStrongNonIntegrable_varphi_i}), if
\begin{equation*}
\Absbig{ ( \varphi_{m_0 - 1} )^{f,\,\CC}_{\xi,\,\xi'} ( v_1(X'),    v_2(X') ) }  \geq 2 \varepsilon d( v_1(X'), v_2(X') )^\alpha,
\end{equation*}
then $\delta_{X'} = 0$, and consequently, by (\ref{eqPfthmPerturbToStrongNonIntegrable_varphi_i}) and (\ref{eqPfthmPerturbToStrongNonIntegrable_UpsilonTempDistr0}), we have
\begin{equation*}
         \Absbig{ ( \varphi_{m_0} )^{f,\,\CC}_{\xi,\,\xi'} ( v_1(X'),    v_2(X') ) }  
=       \Absbig{ ( \varphi_{m_0 - 1} )^{f,\,\CC}_{\xi,\,\xi'} ( v_1(X'),    v_2(X') ) }  
\geq  2 \varepsilon d( v_1(X'), v_2(X') )^\alpha.
\end{equation*}
On the other hand, if
\begin{equation*}
\Absbig{ ( \varphi_{m_0 - 1} )^{f,\,\CC}_{\xi,\,\xi'} ( v_1(X'),    v_2(X') ) }  < 2 \varepsilon d( v_1(X'), v_2(X') )^\alpha,
\end{equation*}
then $\delta_{X'} = 1$ (see (\ref{eqPfthmPerturbToStrongNonIntegrable_deltaX})), and consequently, by (\ref{eqPfthmPerturbToStrongNonIntegrable_varphi_i}), (\ref{eqPfthmPerturbToStrongNonIntegrable_UpsilonTempDistr0}), Property~(a) of the bump functions, Lemma~\ref{lmCellBoundsBM}~(ii), and (\ref{eqDefC26}), we get
\begin{align*}
&                            \Absbig{ ( \varphi_{m_0} )^{f,\,\CC}_{\xi,\,\xi'} ( v_1(X'),    v_2(X') ) }    \\
&\qquad \geq      \AbsBig{ \sum_{j\in\N}   ( \Upsilon_{v_1(\tau_j(X')), \, M' + j}   )^{f,\,\CC}_{\xi,\,\xi'}  ( v_1(X'), v_2(X') ) } 
                            - \Absbig{ ( \varphi_{m_0 - 1} )^{f,\,\CC}_{\xi,\,\xi'} ( v_1(X'),    v_2(X') ) }    \\
&\qquad \geq     \AbsBig{ \sum_{j\in\N}    \Upsilon_{v_1(\tau_j(X')), \, M' + j}  ( v_1 ( \tau_j (X') ) )  } - 2 \varepsilon d( v_1(X'), v_2(X') )^\alpha    \\
&\qquad  =           \sum_{j\in\N} C_{\dagger} \Lambda^{ - \alpha (M'+j) } \varepsilon   -  2 \varepsilon d( v_1(X'), v_2(X') )^\alpha   \\
&\qquad \geq      \Lambda^{-\alpha}  ( 1- \Lambda^{-\alpha} )^{-1} \varepsilon C_{\dagger} C^{-\alpha}   ( \diam_d(X')  )^\alpha  -  2 \varepsilon d( v_1(X'), v_2(X') )^\alpha  \\
&\qquad \geq     2 \varepsilon d( v_1(X'), v_2(X') )^\alpha.
\end{align*}

Hence, we have proved (\ref{eqPfthmPerturbToStrongNonIntegrable_PrePropertyII}). Now we are going to establish (\ref{eqSNIBoundsPerturb}).

Fix arbitrary $N'\geq N_0$. Define $X^{N'+M_0}_{\c,1} \coloneqq \tau_{N'} \bigl( Y^{M_0}_\c \bigr)$ and $X^{N'+M_0}_{\c,2} \coloneqq \tau'_{N'} \bigl( Y^{M_0}_\c \bigr)$ (see also (\ref{eqPfthmPerturbToStrongNonIntegrable_tau})). Note that $\varsigma_1 = \tau_{N'}|_{ Y^{M_0}_\c }$ and $\varsigma_2 = \tau'_{N'}|_{ Y^{M_0}_\c }$.

Then by Lemmas~\ref{lmSnPhiBound},~\ref{lmCellBoundsBM}~(i) and (ii), Proposition~\ref{propCellDecomp}~(i), and Properties~(i) and (iii) in Theorem~\ref{thmPerturbToStrongNonIntegrable},
\begin{align*}
&                       \frac{ \abs{  S_{N' }\phi ( \varsigma_1 (x_1(X_0)) ) -  S_{N' }\phi ( \varsigma_2 (x_1(X_0)) )  -S_{N' }\phi ( \varsigma_1 (x_2(X_0)) ) +  S_{N' }\phi ( \varsigma_2 (x_2(X_0)) )   }  } 
                                  {  d(x_1(X_0),x_2(X_0))^\alpha  }  \\
&\qquad \geq  \frac{ \Absbig{ \phi^{f,\,\CC}_{\xi,\,\xi'} (x_1(X_0)) , x_2(X_0))  }   } { d(x_1(X_0),x_2(X_0))^\alpha }
                                 -  \limsup_{n\to+\infty}       \frac{ \abs{  S_{n - N' }\phi ( \tau_n (v_1(X')) ) -  S_{n - N' }\phi ( \tau_n (v_2(X')) )  }  }   {  \varepsilon^\alpha ( \diam_d (X_0))^\alpha  }  \\
&\qquad\qquad    -  \limsup_{n\to+\infty}       \frac{ \abs{  S_{n - N' }\phi ( \tau'_n (v_1(X')) ) -  S_{n - N' }\phi ( \tau'_n (v_2(X')) )  }  }   {  \varepsilon^\alpha ( \diam_d (X_0))^\alpha  }              \\
&\qquad \geq 2 \varepsilon -  \frac{\Hseminorm{\alpha,\, (S^2,d)}{\phi} C_0}{1-\Lambda^{-\alpha}} \cdot
                                                     \frac{    d ( \tau_{N'} (v_1(X')) , \tau_{N'} (v_2(X')))^\alpha  +  d ( \tau'_{N'} (v_1(X')) , \tau'_{N'} (v_2(X')))^\alpha  }   {  \varepsilon^\alpha ( \diam_d (X_0))^\alpha  }   \\
&\qquad \geq 2 \varepsilon -  \frac{\Hseminorm{\alpha,\, (S^2,d)}{\phi} C_0}{1-\Lambda^{-\alpha}} \cdot
                                                     \frac{   ( \diam_d ( \tau_{N'} (X') ))^\alpha  +  ( \diam_d ( \tau'_{N'} (X') ))^\alpha  }   {  \varepsilon^\alpha ( \diam_d (X_0))^\alpha  }          \\
&\qquad \geq 2 \varepsilon -  \frac{ \bigl(\Hnorm{\alpha}{\varphi}{(S^2,d)} + \varepsilon C_{\sharp} \bigr) C_0}{1-\Lambda^{-\alpha}} \cdot
                                                     \frac{  2 C^\alpha \Lambda^{ - \alpha (M_0 + m_0  N + N' )}  }   {  \varepsilon^\alpha C^{-\alpha} \Lambda^{ - \alpha (M_0 + m_0 N )} }          \\
&\qquad \geq 2 \varepsilon -    2 C^2 \varepsilon^{-\alpha} \bigl(\Hnorm{\alpha}{\varphi}{(S^2,d)} + \varepsilon C_{\sharp} \bigr) C_0    \Lambda^{ - \alpha N_0}  ( 1-\Lambda^{-\alpha} )^{-1}  \\
&\qquad \geq \varepsilon.                                                                                                 
\end{align*}
The last inequality follows from (\ref{eqPfthmPerturbToStrongNonIntegrable_N0}). Property~(ii) in Theorem~\ref{thmPerturbToStrongNonIntegrable} is now established.

\smallskip

The proof of Theorem~\ref{thmPerturbToStrongNonIntegrable} is now complete.
\end{proof}

\begin{rem}
	As remarked by the referee, instead of our more combinatorial approach to construct the bump functions $\Upsilon_{v,n}$ in the proof above, one may approach it from a more metric point of view.  More precisely, as in \cite[(16.7)]{BM17}, one may define a distance function $\rho$ with the modification that only $(n+mN)$-tiles, $m\in\N$, are allowed in any tile chain joining the points, while the weight of an $(n+mN)$-tile is chosen to be $(D_N - 1)^{-(m-1)}$. Then it should follow from the definition that $\rho(x,y) \leq (D_N - 1)^{-(m-1)}$ if $x$ and $y$ are contained in an $(n+mN)$-tile. Moreover, following arguments similar to ones on pp.~319--320 in \cite{BM17}, one should get that $\rho(x,v) \geq 1$ if $v$ is an $(n+N)$-vertex and $x \in S^2 \setminus W^{n+N}(v)$. Then we can alternatively define $\Upsilon_{v,n} \coloneqq \max\{0, \, C_{\dagger} \Lambda^{-\alpha n} (1- \rho(x,v) \}$ and Properties~(a), (b), and~(c) follow.
\end{rem}

\section{Genericity}   \label{sctGeneric}

\begin{proof}[Proof of Theorem~\ref{thmSNIGeneric}]
Note that for each $n\in\N$, the map $F\coloneqq f^n$ is an expanding Thurston map with $\post F = \post f$ and with the combinatorial expansion factor $\Lambda_0(F) = (\Lambda_0(f))^n$ (by (\ref{eqDefCombExpansionFactor}) and Lemma~\ref{lmCellBoundsBM}~(vii)), and  $d$ is a visual metric for $F$ with expansion factor $\Lambda^n$ (by Lemma~\ref{lmCellBoundsBM}). Thus, by \cite[Theorem~15.1]{BM17} (see also Lemma~\ref{lmCexistsL}) and Lemma~\ref{lmSNIwoC}, it suffices to prove Theorem~\ref{thmSNIGeneric} under the additional assumption on the existence of a Jordan curve $\CC \subseteq S^2$ satisfying $\post f \subseteq \CC$ and $f(\CC) \subseteq \CC$. We fix such a curve $\CC$ and consider the cell decompositions induced by the pair $(f,\CC)$ in this proof.

We first show that $\mathcal{S}^\alpha$ is an open subset of $\Holder{\alpha}(S^2,d)$, for each $\alpha\in(0,1]$.

Fix $\alpha\in (0,1]$ and $\phi\in \mathcal{S}^\alpha$ with associated constants $N_0,  \, M_0 \in \N$, $\varepsilon \in (0,1)$, and $M_0$-tiles $Y^{M_0}_\b \in \X^{M_0}_\b$ and $Y^{M_0}_\w \in \X^{M_0}_\w$ as in Definition~\ref{defStrongNonIntegrability}. For each $\c\in\{\b, \, \w\}$, each integer $M\geq M_0$, and each $X\in \X^M$ with $X\subseteq Y^{M_0}_\c$, we choose two points $x_1(X),  \, x_2(X) \in X$ associated to $\phi$ as in Definition~\ref{defStrongNonIntegrability}.

Recall $C_0 > 1$ is the constant depending only on $f$, $\CC$, and $d$ from Lemma~\ref{lmMetricDistortion}.

\smallskip

\emph{Claim.} Fix an arbitrary $\psi \in \Holder{\alpha}(S^2,d)$ with
\begin{equation}   \label{eqPfthmSNIGeneric}
\Hnorm{\alpha}{\phi - \psi}{(S^2,d)} \leq C_0 ( 1- \Lambda^{-\alpha} ) \varepsilon /4.
\end{equation}
Then $\psi$ satisfies Properties~(i) and (ii) in Definition~\ref{defStrongNonIntegrability} with the constant $\varepsilon$ for $\phi$ replaced by $\frac{\varepsilon}{2}$ for $\psi$, and with the same constants $N_0,  \, M_0\in\N$, $M_0$-tiles $Y^{M_0}_\b$, $Y^{M_0}_\w$, and points $x_1(X)$, $x_2(X)$ as those for $\phi$.

\smallskip

Indeed, Property~(i) in Definition~\ref{defStrongNonIntegrability} for $\psi$ follows trivially from that for $\phi$. To establish Property~(ii) for $\psi$, we fix arbitrary integer $N\geq N_0$, and $(N+M_0)$-tiles $X^{N+M_0}_{\c,1},  \, X^{N+M_0}_{\c,2} \in \X^{N+M_0}$ that satisfy (\ref{eqSNIBoundsDefn}) and $Y^{M_0}_\c = f^N\bigl(  X^{N+M_0}_{\c,1}  \bigr) =   f^N\bigl(  X^{N+M_0}_{\c,2}  \bigr)$. Then by (\ref{eqSNIBoundsDefn}), Lemma~\ref{lmSnPhiBound}, and (\ref{eqPfthmSNIGeneric}),
\begin{align*}
&                           \abs{  S_{N }\psi ( \varsigma_1 (x_1(X)) ) -  S_{N }\psi ( \varsigma_2 (x_1(X)) )  -S_{N }\psi ( \varsigma_1 (x_2(X)) ) +  S_{N }\psi ( \varsigma_2 (x_2(X)) )   }   \\
&\qquad       \geq   \abs{  S_{N }\phi ( \varsigma_1 (x_1(X)) ) -  S_{N }\phi ( \varsigma_2 (x_1(X)) )  -S_{N }\phi ( \varsigma_1 (x_2(X)) ) +  S_{N }\phi ( \varsigma_2 (x_2(X)) )   }  \\
&\qquad\qquad   -  \sum_{i\in\{1, \, 2\}}       \abs{  S_{N } (\psi - \phi) ( \varsigma_i (x_1(X)) )  -S_{N } (\psi - \phi) ( \varsigma_i (x_2(X)) )  }    \\
&\qquad       \geq  d(x_1(X),x_2(X))^\alpha   \bigl( \varepsilon -  2  \Hseminorm{\alpha,\, (S^2,d)}{\psi - \phi} C_0 ( 1-\Lambda^{-\alpha} )^{-1} \bigr)       \\
&\qquad       \geq   d(x_1(X),x_2(X))^\alpha    \varepsilon   /  2     .
\end{align*}
The claim is now established.

Hence, $\mathcal{S}^\alpha$ is open in $\Holder{\alpha}(S^2,d)$.

Finally, recall that $1<\Lambda \leq \Lambda_0(f)$ (see \cite[Theorem~16.3]{BM17}). Thus, if either $\alpha \in (0,1)$ or $\Lambda \neq \Lambda_0(f)$, then $\Lambda^{\alpha} < \Lambda_0(f)$, and the density of $\mathcal{S}^\alpha$ in $\Holder{\alpha}(S^2,d)$ follows immediately from Theorem~\ref{thmPerturbToStrongNonIntegrable}.
\end{proof}

\end{document}